\documentclass{article}
\usepackage[utf8]{inputenc}
\usepackage{graphicx} 

\usepackage[utf8]{inputenc}
\usepackage[T1]{fontenc}
\usepackage{amsmath}
\usepackage{amssymb}
\usepackage{amsthm}
\usepackage{amsfonts}
\usepackage[english]{babel}
\usepackage{lmodern}
\usepackage{mathrsfs}
\usepackage{csquotes}
\usepackage{enumitem}
\usepackage{array}
\usepackage{amsfonts}
\usepackage{mathtools}
\usepackage{esint}
\usepackage{array}
\usepackage{authblk}
\usepackage{nccmath}
\usepackage{bold-extra}
\usepackage{cite}
\usepackage{tikz}
\usepackage{overpic}
\usepackage{geometry}
\usepackage{placeins}
\usepackage{hyperref}
\usetikzlibrary{hobby}
\usetikzlibrary{backgrounds}

\usepackage{pgf}
\usepackage{pgfplots}
\pgfplotsset{compat=newest}
\usepackage{xcolor}
\usetikzlibrary{plotmarks,backgrounds,patterns}
\usetikzlibrary{external}
\tikzset{external/only named=true}

\usepackage{selinput}

\usepackage{todonotes}

\usepackage{dsfont}

\theoremstyle{plain}

\newtheorem{theorem}{Theorem}[section]
\newtheorem{proposition}[theorem]{Proposition}
\newtheorem{lemma}[theorem]{Lemma}
\newtheorem{corollary}[theorem]{Corollary}

\theoremstyle{remark}

\newtheorem{definition}[theorem]{Definition}
\newtheorem{remark}[theorem]{Remark}
\newtheorem{example}[theorem]{Example}

\numberwithin{equation}{section}

\newcommand{\norm}[1]{\left\lVert#1\right\rVert}

\newcommand{\RR}{\mathbb{R}}

\newcommand{\NN}{\mathbb{N}}

\newcommand{\tr}{\textup{tr}}
\newcommand{\sym}{\textup{sym}}

\newcommand{\el}{\textup{el}}

\newcommand{\n}{\mathbf{n}}

\newcommand{\op}{\textup{op}}

\newcommand{\Div}{\textup{div}}

\newcommand{\dist}{\textup{dist}}

\newcommand{\id}{\mathbf{id}}
\newcommand{\SO}[1]{\textup{SO}(#1)}

\newcommand{\GLp}[1]{\textup{GL}_{+}(#1)}

\newcommand{\RNum}[1]{\uppercase\expandafter{\romannumeral #1\relax}}

\newcommand{\address}[1]{%
  \bgroup
  \renewcommand\thefootnote{} 
  \footnotetext{#1} 
  \egroup
}

\title{Local Well-Posedness of a Model for Stress-Driven Growth in the Presence of Nutrients}

\author[1]{Helmut Abels}
\author[2]{Julian Blawid}
\author[3]{Georg Dolzmann}

\affil[1,2,3]{Faculty of Mathematics, University of Regensburg, Universit\"atsstr. 31, 93053 Regensburg, Germany}

\date{\today}

\begin{document}

\maketitle

\begin{abstract}
    A model for morphoelastic growth, that is, growth influenced by elastic stress, driven by the absorption of nutrients is considered. The model features a multiplicative decomposition of the deformation gradient into an elastic contribution and a growth tensor. While the evolution of the system is governed by an ordinary differential equation for the growth tensor on a suitable Banach space, which depends on the elastic stresses and the concentration of a nutrient field, the total deformation is given by the solution of a quasi-static equilibrium equation arising from the formal Euler-Lagrange equations of a hyperelastic variational integral. The nutrient concentration is determined by a linear elliptic reaction-diffusion equation which is formulated in Lagrangian coordinates and whose coefficients depend on the growth tensor as well as the deformation gradient accounting for the change of material properties due to growth and elastic deformation. Existence and uniqueness of solutions of this fully coupled system of differential equations is proved via a fixed-point argument. 
\end{abstract}



\address{{\bf Corresponding author:} Georg Dolzmann (e-mail: georg.dolzmann@mathematik.uni-regensburg.de)}

\section{Introduction}

In this contribution, we consider a model for morphoelastic growth, that is, growth under the influence of elastic stress, in the presence of nutrients. We adopt a macroscopic description of the involved physical processes within the framework of continuum mechanics. While results in the large-strain setting typically rely on various regularizations or specific structural assumptions, our approach leverages the small-strain regime to present a nonlinear model that requires no regularization at all. We will show existence and uniqueness of solutions on a short time interval. Let $d \geq 2$ and let $\Omega \subset \RR^d$ denote a domain modeling a body of tissue in its unstressed reference configuration. Our model features a multiplicative decomposition of the deformation gradient associated to the total deformation $y:[0,T] \times \Omega \to \RR^d$ into an elastic part, denoted by $F_{\el}:[0,T] \times \Omega \to \GLp{d}$, and a growth-related part, denoted by $G:[0,T] \times \Omega \to \GLp{d}$, henceforth called the growth tensor,
\begin{equation}\label{equation multiplicative decomposition}
    \nabla y = F_{\el}G.
\end{equation}
This multiplicative decomposition was first introduced in \cite{Kroener.1960, Lee.1969} in the context of crystal plasticity and is by now classical in large-strain growth models. The multiplicative decomposition describes the modeling assumption that the local deformation, which is encoded by the deformation gradient as a linear mapping from the tangent bundle of the reference manifold, $T\Omega$, into the tangent bundle of the deformed manifold $y(t,\Omega)$ at time $t \in [0,T]$, can be written, in a point $x \in \Omega$, as the concatenation of the linear mapping $G(t,x)[\cdot]$, describing the local deformation due to pure, unconstrained growth and the linear mapping $F_{\el}(t,x)[\cdot]$, describing the local deformation due to elastic deformation. In case that the growth tensor $G$ is compatible, i.e., in case that there exist a mapping $g:[0,T] \times \Omega \to \RR^d$ with $G=\nabla g$, describing the deformation due to pure growth, an intermediate configuration, describing the configuration after pure, unconstrained growth, exists in the form of $g(\cdot,\Omega)$. In this case, the elastic part of the deformation gradient is compatible as well, that is, there exists a map $\phi(t,\cdot):g(t,\Omega) \to \RR^d$ with $\nabla \phi \circ g = F_{\el}$ and the multiplicative decomposition \eqref{equation multiplicative decomposition} is implied by the chain rule. 
\newline

Our model involves various physical processes which act on vastly different time-scales. In particular, the growth process happens on a much longer time-scale than the diffusion and absorption of nutrients, which, in turn, happen on a much longer time-scale than the mechanical response. Therefore, inertial forces can be considered negligible. Hence, we take a quasi-static approach and assume that the nutrient concentration as well as the mechanical state of the system are in equilibrium at all times.
\newline

The multiplicative decomposition allows us to formulate the quasi-static equilibrium equations in Lagrangian coordinates by means of the first Piola transform. We assume that the body $\Omega$ is subject to a hyperelastic material model. Moreover, we impose Dirichlet boundary conditions on a relatively open subset $\Gamma_D\subset \partial\Omega$ of the boundary with positive surface measure and inhomogeneous Neumann boundary conditions on a separated part $\Gamma_N \subset \partial\Omega$ of the boundary, which is also assumed to be relatively open in $\partial\Omega$. Let $W: \Omega \times \GLp{d} \to \RR_{\geq 0}$ be a sufficiently smooth elastic energy density and let $D_pW$ denote the partial derivative of $W$ with respect to its second variable. Given a growth tensor $G=G(t,\cdot):\Omega \to \GLp{d}$ at time $t \in [0,T]$ and a deformation $y=y(t,\cdot):\Omega \to \RR^d$ with, for some constant $c>0$, $\det(\nabla y G^{-1}) \geq c$, we denote by
\begin{equation*}
    P(t):=\det(G)D_pW(\cdot,\nabla y G^{-1})G^{-T}
\end{equation*}
the first Piola-Kirchhoff stress tensor. Let $f:\Gamma_D \to \RR^d$ denote a superimposed displacement on the Dirichlet part of the boundary and let $g:\Gamma_N \to \RR^d$ denote a superimposed boundary traction on the Neumann part of the boundary. The total deformation at time $t \in [0,T]$ is assumed to be given by a solution of the quasi-static equilibrium equation
\begin{equation}
    \begin{cases}\label{equation equilibrium equation introduction}
        -\Div(P(t))=0_d & \textup{on } \Omega
        \\ y(t)=f & \textup{on } \Gamma_D
        \\ P(t)\n=g & \textup{on } \Gamma_N. 
    \end{cases}
\end{equation}
This equation corresponds to the pullback of the quasi-static equilibrium equation from the virtual configuration, which can only be realized in case the growth tensor is compatible, to the Lagrangian coordinate frame by taking into account the multiplicative decomposition.
A major obstacle for the analysis in the theory of nonlinear elasticity is the absence of a variational formulation of this equation under physically reasonable assumptions on the hyperelastic energy density $W$. In particular, it is often assumed that the energy blows up if the local volume ratio, represented by the determinant of the elastic part of the deformation gradient, $\det(F_{\el})=\det(\nabla y G^{-1})$, approaches zero. Under these circumstances, it is largely unknown, whether the energy is Gateaux differentiable at a given minimizer. However, because in our model the boundary data and the pertubation of the energy by the growth tensor force the total deformation to be close to the identity on a short time interval, we can show existence of stationary solutions via Banach's fixed-point theorem. Note, that no convexity assumptions on $W$ are required. However, our model is compliant with the standard assumptions of frame-indifference and the aforementioned blow-up of the energy. 
\newline

In many applications, growth processes are driven by the absorption of nutrients. As diffusion and absorption of nutrient happen on much shorter time-scales than the growth process, we assume that the diffusion and absorption of nutrients is a quasi-static process. Our model is complemented by a linear elliptic reaction-diffusion equation modeling the diffusion and absorption of nutrients. Moreover, elastic compression generally changes the material properties, whereas pure, unconstrained growth does not. This is reflected in our model by the assumption that the diffusion rate $D(G,\nabla y)$ and the reaction rate $\beta(G,\nabla y)$ depend on a given growth tensor $G$ and, for a given deformation $y$, the total deformation gradient $\nabla y$. Given a growth tensor $G$ and the total deformation $y(G)$, the nutrient concentration $N(G)$ is determined by the equation
\begin{equation}\label{equation nutrient equation introduction}
    \begin{cases}
        -\sum_{i,j=1}^d \partial_i(D(G,\nabla y(G))_{ij} \partial_j N(G)) + \beta(G,\nabla y(G))N(G)=0 & \textup{on } \Omega
        \\ N(G) = f_n & \textup{on } \Gamma_D^n
        \\ D(G,\nabla y(G)) \nabla N(G) \cdot \n = g_n & \textup{on } \Gamma_N^n.
    \end{cases}
\end{equation}

The evolution of the system is driven by an ordinary differential equation for the growth tensor $G$, which depends on the elastic stress implicitly by means of the deformation gradient. This will require us to prove that the deformation gradient depends Lipschitz-continuously on the growth tensor to apply Picard-Lindel\"of's theorem. In the quasi-static situation, proving time regularity of the first Piola-Kirchhoff stress tensor poses the main obstacle in the analysis. This is due to the fact that for a lack of convexity, a variational formulation, which would allow us to prove existence of stationary solutions, does not suffice to prove time-regularity of the first Piola-Kirchhoff stress tensor. Let $G_0 \in C^{1,\mu}(\overline{\Omega},\GLp{d})$, close to the identity, denote the initial data, let $\mathcal{G}:\GLp{d} \times \GLp{d} \times \RR_+ \times \Omega \to \GLp{d}$, let $y(G) \in C^{2,\mu}(\Omega,\RR^{d})$ denote the deformation associated to the growth tensor $G$ and let $N(G) \in C^{2,\mu}(\overline{\Omega})$ denote the nutrient concentration associated to the growth tensor $G$. The ordinary differential equation governing the evolution of the growth tensor is assumed to be given by
\begin{equation}\label{equation ODE introduction}
    \begin{cases}
        \frac{d}{dt}G(t)(x)=\mathcal{G}(G(t)(x),\nabla y(G(t))(x),N(G(t))(x),x) & \textup{for all } (t,x) \in [0,T] \times \Omega
        \\ G(0)(x)=G_0(x) & \textup{for all } x \in \Omega.
    \end{cases}
\end{equation}
For simplicity, we frequently drop the dependence on $x$. In many applications, growth is facilitated or inhibited by elastic stress, cf. \cite{Jones.Chapman.2012}. A dependence of the ordinary differential equation \eqref{equation ODE introduction} on the first Piola-Kirchhoff stress tensor $P(t)$ is covered by our regularity assumptions on the elastic energy density and the growth law $\mathcal{G}$, cf. \textbf{(W1)} and \textbf{(G1)} in Chapter \ref{chapter main theorem} by means of the dependence of $\mathcal{G}$ on the growth tensor $G$ and the deformation gradient $\nabla y$. 
\newline

An important question in models involving a multiplicative decomposition of the deformation gradient is the correct interpretation of the intermediate configuration because it generally does not exist on a global scale except for some special cases. Many physically measurable quantities, such as local distances and angles after growth, depend only on the symmetric right Cauchy-Green growth tensor $G^TG$
which can be interpreted as a Riemannian metric on the reference configuration by means of the bilinear form $g(v,w)=(Gv)^TGw$. If the intermediate configuration exists, this Riemannian metric on the reference configuration corresponds to the pullback of the Euclidean metric on the intermediate configuration onto the reference configuration. The incompatibility of the growth tensor is then encoded by the curvature of the Riemannian metric induced by $G^TG$. However, in anisotropic materials, such as muscle tissue, the orientation of the intermediate frame is physically tied to the underlying microstructure. To model anistropic growth, a formulation of the ordinary differential equation governing the growth hinges upon the use of the full growth tensor as the skew-symmetric part of the intermediate velocity gradient, which is given by $\frac{1}{2}(\dot{G}G^{-1}-(\dot{G}G^{-1})^T)$, encodes the rate of rigid rotations due to growth of the microstructure relative to the bulk material. An analogous isotropic model, in which the evolution is driven by an ordinary differential equation for the Riemannian metric $G^TG$, respecting gauge invariance, can be formulated as well and existence of solutions can be proven similarly. In this isotropic model, the ordinary differential equation for the growth tensor is compatible with the principle of gauge invariance if it only depends on the principle stress invariants and not on the total deformation gradient.
\newline

To show existence of solutions at small strains of the system consisting of the equilibrium equation \eqref{equation equilibrium equation introduction}, the nutrient equation \eqref{equation nutrient equation introduction} and the ordinary differential equation \eqref{equation ODE introduction}, we largely rely on classical methods. In the first step, we show with Banach's fixed-point theorem that, for a given growth tensor which is sufficiently close to the identity, the equilibrium equation \eqref{equation equilibrium equation introduction} admits a unique solution. This is achieved by linearizing the equilibrium equation \eqref{equation equilibrium equation introduction} around the identity and applying a Korn equality, which was proven in \cite{Pompe.2003} and which is based on the earlier work \cite{Neff.2002}, to show that the linearized operator is uniformly elliptic. We then exploit the Schauder theory developed in \cite{Agmon.Douglis.Nirenberg.1959, Agmon.Douglis.Nirenberg.1964, Douglis.Nirenberg.1955} and show with the implicit function theorem that the deformation associated to a growth tensor depends Lipschitz-continuously on the growth tensor. In the second step, we show that, for a given growth tensor close to the identity and the associated mechanical deformation, a solution to the nutrient equation \eqref{equation nutrient equation introduction} exists and depends Lipschitz-continuously on the growth tensor. We conclude by applying Picard-Lindel\"of's theorem to show that a solution to the ordinary differential equation \eqref{equation ODE introduction} exists on a short time interval.
\newline

Our result can be generalized to the large-strain setting. In fact it suffices to assume that, for an initial growth tensor $G_0$, a displacement $f$ on $\Gamma_D$ and traction forces $g$ on $\Gamma_N$, the initial configuration $y_0$ is close to a local minimum of the energy if the energy satisfies a coercivity condition close to this particular local minimum. However, for notational reasons, we focus on the setting where the traction forces are close to zero, the displacement is close to the identity, the growth tensor is close to the identity and the reference configuration is unstressed.
\newline

\textbf{Related Literature:} In recent years, there has been a burgeoning interest within the mathematical community regarding the rigorous analytical treatment of models describing various phenomena from the broad area of biomechanics and models for morphoelastic growth in particular. However, the transition from phenomenological description to rigorous analysis reveals a multitude of challenges mainly originating from the previously discussed difficulties in the theory of finite elasticity. Whereas for comparable models from the very active research field of large-strain elastoplasticity, significant analytical progress has been obtained, for example in \cite{Mielke.2003} and in \cite{Mielke.Rossi.Savare.2018}, these findings cannot be translated into our setting. The reason for this is twofold, the ordinary differential equation in our model does not admit a variational formulation and, moreover, the coupling in the ordinary differential equation requires the elastic stresses to be Lipschitz-continuous in time. For a model describing morphoelastic growth, a first rigorous result has been obtained in \cite{Davoli.Nik.Stefanelli.2023}, in which the authors proved existence of solutions at large strains for a regularized version of the model we are also treating. Moreover, in \cite{Bangert.Dolzmann.2023, Blawid.Dolzmann.2026}, the one-dimensional version of the same model is discussed. However, the proof of the existence result in the latter two contributions heavily relies on the fact that, in one dimension, as observed in \cite{Ball.1981}, minimizers of hyperelastic variational integrals with a convex integrand are continuously differentiable. Our result can be compared to the result obtained in \cite{Neff.2005} in which the author showed a local well-posedness result for a similar system of equations in the context of elasto-plasticity under strong structural assumptions regarding the elastic energy density. In particular, the author assumes in \cite{Neff.2005}, that the elastic energy is given by a Saint-Venant Kirchhoﬀ energy, which circumvents many analytical issues we face. Furthermore, models for accretive growth have been discussed in \cite{Davoli.Nik.Stefanelli.Tomasetti.2024, Chiesa.Stefanelli.2025, Chiesa.Stefanelli.2026}. In \cite{Yang.Jaeger.Neuss-Radu.Richter.2016}, a model for plaque growth in blood vessels was introduced. This model was further analyzed in \cite{Abels.Liu.2023.2, Abels.Liu.2023.1, Abels.Liu.2023.3}. Links between the microscopic and macroscopic description of tissue-growth in plants were studied analytically and numerically in the context of homogenization in \cite{Boudaoud.Kiss.Ptashnyk.2023, Piatniski.Ptashnyk.2020}. In \cite{Lewicka.Raoult.2018, Lewicka.Raoult.Ricciotti.2017, Lewicka.2020, Lewicka.Lucic.2020, Han.Lewicka.2024, Lewicka.Oocha.Pakzad.2015, Lewicka.Pakzad.2013, Lewicka.2011, Lewicka.Mahadevan.Pakzad.2011, Bhattacharya.Lewicka.Schaffner.2016} the authors consider a variety of models, in particular also plate and shell models, in the context of nonlinear elasticity, for pre-strained materials. The pre-strain is induced by a prescribed Riemannian metric on the reference configuration, corresponding to the quantity $G^TG$ in our model, which may originate from morphogenetic processes in biological tissues. A thermodynamically consistent viscoelastic model at large strains under the assumption of a Korn-type condition on the viscous stresses with such a prescribed Riemannian metric on the reference configuration, potentially governing a growth process, was treated in \cite{Lewicka.Mucha.2013}. In \cite{Lewicka.Mucha.2016}, a class of stress-assisted diﬀusion systems, in which the growth rate can be interpreted as the solution of a diffusion equation, is treated in the hyperbolic as well as the quasi-static setting. Moreover, in \cite{Bressan.Lewicka.2018}, a quasi-static model for morphogenesis is presented in the context of a free boundary problem in which the growth process is driven by the presence of a morphogen. In the context of second-gradient viscoelastic materials, a growth model, in which the growth process is induced by a phase-field variable subject to a Cahn-Hilliard model, was introduced in \cite{Eiter.Schmeller.2025}. Additionally, there is an extensive body of literature from the modeling community concerning models for stress-modulated growth, summarized to a large extend in the monograph \cite{Goriely.2017}. Models in which the growth process and the elastic deformation were treated as two distinct sub-processes have been established in \cite{Ambrosi.Ateshian.Arruda.Cowin.Dumais.Goriely.Holzapfel.Humphrey.Kemkemer.Kuhl.Olberding.Taber.Garipikati.2011, Ambrosi.Guillou.2007, Ambrosi.Mollica.2002}. Growth models involving a multiplicative decomposition of the deformation gradient into an elastic part and a growth tensor were treated for example in \cite{Lubarda.Hoger.2002, Menzel.Kuhl.2021, Rodriguez.Hoger.McCulloch.1994, Skalak.Dasgupta.Moss.Otten.Dullemejer.Vilman.1982, Taber.1995, Huang.Ogden.Penta}. Buckling, bending and twisting phenomena are discussed in \cite{Bressan.Palladino.Shen.2017}. More recently, in  \cite{Erlich.Zurlo.2024, Erlich.Zurlo.2025, Erlich.Harmansa.2025}, the authors present models for morphoelastic growth and discuss several ramifications of the incompability of the growth tensor. 

\section{Main Theorem}\label{chapter main theorem}

\textbf{General assumptions and notation:} Let $d \geq 2$, $\mu \in (0,1]$ and let $\Omega \subset \RR^d$ be open, bounded and connected with $C^{2,\mu}$-boundary. For $k,n \in \NN_0$ and $n \geq 1$, we denote by $C^{k,\mu}(\overline{\Omega},\RR^n)$ the classical H\"older spaces and, for $\Gamma \subset \partial \Omega$ relatively open and non-empty, with $\partial\Gamma = \emptyset$, we denote by $C^{k,\mu}(\overline{\Gamma},\RR^n)$ the respective trace space defined via charts as usual. We assume that $\Gamma_D$ and $\Gamma_N$ are relatively open subsets of $\partial \Omega$ with $\partial\Omega = \Gamma_D \cup \Gamma_N$, $\Gamma_D \cap \Gamma_N = \emptyset$ and with $\Gamma_D \neq \emptyset$. In particular, $\Gamma_D$ and $\Gamma_N$ are closed subsets of $\RR^d$. Later, we will prescribe Dirichlet boundary conditions on $\Gamma_D$ and Neumann boundary conditions on $\Gamma_N$ for the total deformation of the body. To prescribe boundary conditions for the nutrient field, we assume that $\Gamma^n_D$ and $\Gamma^n_N$ are relatively open subsets of $\partial \Omega$ with $\partial\Omega = \Gamma^n_D \cup \Gamma^n_N$ and with $\Gamma^n_D \cap \Gamma^n_N = \emptyset$. In particular, $\Gamma^n_D$ and $\Gamma^n_N$ are closed subsets of $\RR^d$. For $\Gamma \subset \partial\Omega$ we will denote by $C_{0,\Gamma}^{2,\mu}(\overline{\Omega},\RR^d)$ the space
\begin{equation*}
    C_{0,\Gamma}^{2,\mu}(\overline{\Omega},\RR^d):=\{u \in C^{2,\mu}(\overline{\Omega},\RR^d):\, u=0 \textup{ on } \Gamma\}.
\end{equation*}
We denote by $\mathbf{1}_d$ the unit matrix and by $\mathbf{0}_d$ the matrix that is $0$ entry-wise. By $\SO{d}:=\{Q \in \RR^{d \times d};\, Q^TQ=\mathbf{1}_d,\, \det(Q)=1\}$ we denote the set of all rotations and by $\GLp{d}:=\{A \in \RR^{d \times d}:\, \det(A)>0\}$ we denote the set of orientation preserving isomorphisms on $\RR^d$. In the following, we will use Einstein's summation convention, i.e., we will sum over all indices which appear once as an upper and once as a lower index in an expression. We adapt the following notation for open balls. For $i \in \NN$, $R_G, R_G^i, R_y, R_y^i, R_f, R_g > 0$ and $\rho \in \{R_G, R_G^i, R_y, R_y^i, R_f, R_g\}$, let $B_{\rho}(\cdot)$ denote the open ball of radius $\rho$ in the following respective spaces:
\begin{align*}
    & B_{R_G}(\cdot), B_{R_G^i}(\cdot), B_{\tilde{R}_G^i}(\cdot)\subset C^{1,\mu}(\overline{\Omega}, \mathbb{R}^{d \times d}), \quad B_{R_y}(\cdot), B_{R_y^i}(\cdot), B_{\tilde{R}_y^i}(\cdot) \subset C^{2,\mu}_{0,\Gamma_D}(\overline{\Omega}, \mathbb{R}^d),
    \\ & \quad B_{R_f}(\cdot), B_{R_f^i}(\cdot), B_{\tilde{R}_f^i}(\cdot) \subset C^{2,\mu}(\Gamma_D, \mathbb{R}^d), \quad B_{R_g}(\cdot) \subset C^{1,\mu}(\Gamma_N, \mathbb{R}^d).
\end{align*}
All balls are defined with respect to the standard H\"older norms on their ambient spaces. For an open ball $B_{\rho}(\cdot)$ of radius $\rho >0$, we denote its closure in its ambient space by $\overline{B}_{\rho}(\cdot)$. We will also denote balls around the displacement $u$ by $B_{R_y}(\cdot), B_{R_y^i}(\cdot)$ and $B_{\tilde{R}_y^i}(\cdot)$ because we will show the existence of a displacement in a ball around $0_d$ for which the associated deformation solves our system. The deformation will then lie in a ball of the same radius around the identity.

We define, for $F \in \GLp{d}$ and $R>0$, the open ball $B_R(F) \subset \GLp{d}$ with respect to the $\infty$-norm in $\RR^{d \times d}$. Moreover, we denote by $W:\overline{\Omega} \times \GLp{d} \to \RR_{\geq 0}$ a stored elastic energy density. For $k \in \NN$, $D^k_pW$ stands for the $k$-th partial derivative of $W$ with respect to its second variable. We assume that $W$ satisfies the following assumptions \textbf{(W1)}-\textbf{(W3)}.

\begin{enumerate}
    \item[\textbf{(W1)}] Regularity: Let $R>0$ such that $\overline{B}_R(\mathbf{1}_d) \subset \GLp{d}$. We assume that $W \in C^5(\overline{\Omega} \times B_R(\mathbf{1}_d),\RR_{\geq 0})$.
    
    \item[\textbf{(W2)}] Coercivity: Let $R>0$ be as in \textbf{(W1)}. There exists a constant $c_W>0$ such that for all $F \in B_R(\mathbf{1}_d)$ and for all $x \in \Omega$
    \begin{equation*}
        W(x, F) \geq c_W \dist(F,\SO{d})^2.
    \end{equation*}
    
    \item[\textbf{(W3)}] Unstressed initial configuration: For all $x \in \overline{\Omega}$, $W(x,\mathbf{1}_d) = 0$.

\end{enumerate}

\begin{remark}\label{remark frame indifference}
    If the energy density $W$ is defined on the entire Lie group $\GLp{d}$, the aforementioned conditions \textbf{(W1)}-\textbf{(W3)} are compliant with the principle of frame-indifference, that is, for all $Q \in \SO{d}$ and for all $F \in \GLp{d}$
    \begin{equation*}
        W(\cdot, QF) = W(\cdot, F)\qquad \textup{on }\Omega.
    \end{equation*}
\end{remark}

\begin{remark}
    If $W$ is defined on $\GLp{d}$, a standard assumption in the theory of finite elasticity is the assumption that the energy density blows up as the volume ratio approaches zero, that is
    \begin{equation*}
        W(\cdot,F)\xrightarrow[\det(F)\searrow 0]{} +\infty.
    \end{equation*}
    As we are presenting a nonlinear model for which the solutions will lie in a sufficiently small ball around the identity, this assumption is not imposed although it is compliant with assumptions \textbf{(W1)}-\textbf{(W3)} and frame indifference, cf. Remark \ref{remark frame indifference}, as illustrated in Example \ref{example}. 
\end{remark}

To prove existence of solutions locally in time, we will resort to Picard-Lindel\"of's theorem. Therefore, we require the right-hand side of the ordinary differential equation governing the evolution of the growth tensor to be locally Lipschitz continuous and we will assume the following.

\begin{itemize}
    \item[\textbf{(G1)}] Regularity: Let $R>0$ be as in \textbf{(W1)}. We assume that $\mathcal{G} \in C^2(B_R(\mathbf{1}_d) \times B_R(\mathbf{1}_d) \times \RR_{\geq 0} \times \overline{\Omega})$ and for all $x \in \overline{\Omega}$, $\mathcal{G}(\cdot,\cdot, \cdot,x) \in C^3(B_R(\mathbf{1}_d) \times B_R(\mathbf{1}_d) \times \RR_{\geq 0})$.
\end{itemize}

Moreover, we will make the following assumptions regarding the coefficients of the nutrient equation. 

\begin{itemize}
    \item[\textbf{(N1)}] Regularity and ellipticity of the coefficients: For $R>0$ as in \textbf{(W1)}, $R_G^1, R_{\nabla y}^1 \in (0, R)$ and $B_{R_{\nabla y}^1}(\mathbf{1}_d) \subset C^{1,\mu}(\overline{\Omega},\RR^{d \times d})$, we assume that the diffusion rate $D$ satisfies
    \begin{equation*}
        D \in C^1(B_{R_G^1}(\mathbf{1}_d) \times B_{R_{\nabla y}^1}(\mathbf{1}_d), C^{1,\mu}(\overline{\Omega},\RR^{d \times d}_{\sym}))
    \end{equation*}
    and the absorption rate $\beta$ satisfies
    \begin{equation*}
        \beta \in C^1(B_{R_G^1}(\mathbf{1}_d) \times B_{R_{\nabla y}^1}(\mathbf{1}_d), C^{0,\mu}(\overline{\Omega})).
    \end{equation*}
    Moreover, we assume that there exists a constant $\nu>0$ such that for all $G \in B_{R_G^1}(\mathbf{1}_d)$, for all $Y \in B_{R_{\nabla y}^1}(\mathbf{1}_d)$ and for all $\xi \in \RR^d$
    \begin{equation*}
        \sum_{i,j=1}^d D(G,Y)_{ij}\xi_i\xi_j \geq \nu|\xi|^2
    \end{equation*}
    and that, for all $G \in B_{R_G^1}(\mathbf{1}_d)$ and for all $Y \in B_{R_{\nabla y}^1}(\mathbf{1}_d)$, $\beta(G,Y) \geq 0$ on $\overline{\Omega}$.
    \item[\textbf{(N2)}] Boundary conditions: We assume that $f_n \in C^{2,\mu}(\Gamma_D^n)$, $g_n \in C^{1,\mu}(\Gamma_N^n)$, $f_n \geq 0$ on $\Gamma_D^n$ and $g_n \geq 0$ on $\Gamma_N^n$. Furthermore, we assume that $\Gamma^n_D \neq \emptyset$ or that, for $R_G^1, R_{\nabla y}^1$ as in \textbf{(N1)} and for all $G \in B_{R_G^1}(\mathbf{1}_d)$ and $Y \in B_{R_{\nabla y}^1}(\mathbf{1}_d)$, $\beta(G,Y) > 0$ on a subset of $\overline{\Omega}$ with positive measure.
\end{itemize}

\begin{remark}
    Note that the physical normal outward flux is given by $-g_n$ and that \textbf{(N2)} allows us to apply the maximum principle and Hopf's lemma to show that the nutrient concentration remains non-negative. 
\end{remark}

\begin{remark}\label{remark pointwise definition coefficients nutrient equation}
    The diffusion rate and the absorption rate, cf. \textbf{(N1)}, can be also defined pointwise. In fact it suffices to define $D \in C^2(B_{R}(\mathbf{1}_d) \times B_{R}(\mathbf{1}_d), \RR^{d \times d}_{\sym})$ and $\beta \in C^1(B_{R}(\mathbf{1}_d) \times B_{R}(\mathbf{1}_d), \RR_{\geq 0})$, where $B_R(\mathbf{1}_d) \subset \RR^{d \times d}$. For a given growth tensor $G \in C^{1,\mu}(\overline{\Omega},\RR^{d \times d})$ and the associated deformation $y(G) \in C^{2,\mu}(\overline{\Omega},\RR^d)$, the diffusion rate and the absorption rate in our model are then given by the concatenation $D(G,\nabla y(G)) \in C^{1,\mu}(\overline{\Omega},\RR^{d \times d})$ and $\beta(G,\nabla y(G)) \in C^{0,\mu}(\overline{\Omega})$ respectively.
\end{remark}

\begin{remark}\label{remark frame indifference 2}
    In order to obtain a physically reasonable model, one has to account for objectivity of the ordinary differential equation and the coefficients in the nutrient equation as well. In particular, it is desirable to assume that, if $\mathcal{G}$ is defined on $\GLp{d} \times \GLp{d} \times \RR_{\geq 0} \times \overline{\Omega}$, for all $Q \in \SO{d}$, $G \in \GLp{d}$, $Y \in \GLp{d}$, $N \in \RR_{\geq 0}$,
    \begin{equation*}
        \mathcal{G}(G, QY, N, \cdot)=\mathcal{G}(G, Y, N, \cdot).
    \end{equation*}
    This is consistent with Assumption \textbf{(G1)} and already given, if $\mathcal{G}$ only depends on the first Piola-Kirchhoff stress tensor, which is, for a given growth tensor $G$ and a deformation $y$, defined by 
    \begin{equation*}
        P(G):=\det(G)D_pW(\cdot,\nabla y G^{-1})G^{-T},
    \end{equation*}
    instead of the full deformation gradient due to Remark \ref{remark frame indifference}. Additionally, if the diffusion rate $D$ and the absorption rate $\beta$ are defined pointwise on $\GLp{d}\times \GLp{d}$, frame indifference of the entire model is guaranteed if we assume that for all $Q \in \SO{d}$, for all $G \in \GLp{d}$ and for all $F \in \GLp{d}$
    \begin{equation*}
        D(G,QF)=D(G,F)\quad \textup{and}\quad \beta(G,QF)=\beta(G,F).
    \end{equation*}
\end{remark}

\begin{definition}
    For fixed $G \in C^{1, \mu}(\overline{\Omega}, \overline{B}_R(\mathbf{1}_d))$ we define $\mathcal{A}_G: C^{1,\mu}(\overline{\Omega},\RR^{d \times d}) \to C^{1, \mu}(\overline{\Omega}, \RR^{d \times d})$ via
    \begin{equation}\label{equation operator quasi-static momentum balance equation}
        \mathcal{A}_G(Y) = \det(G)D_pW(\cdot,Y G^{-1})G^{-T} \qquad\forall Y \in C^{1,\mu}(\overline{\Omega},\RR^{d \times d}).
    \end{equation}
\end{definition}

We note that the regularity of $\mathcal{A}_G(Y)$ follows by Propsition \ref{proposition composition in Holder spaces} below.

\begin{definition}\label{definition definition of solution}
    Let $T>0$, $G_0 \in C^{1,\mu}(\overline{\Omega}; \RR^{d \times d})$, $f \in C^{2,\mu}(\Gamma_D,\RR^d)$ and $g \in C^{1,\mu}(\Gamma_N,\RR^d)$. We say that the triple $(G, y, N)$ with $G \in C^1([0,T]; C^{1, \mu}(\overline{\Omega}, \RR^{d \times d}))$, $y \in C^1([0,T];C^{2,\mu}(\overline{\Omega},\RR^d))$ and $N \in C^1([0,T];C^{2,\mu}(\overline{\Omega}))$ is a local solution to the morphoelastic growth problem on the time interval $[0,T]$ if for all $t \in [0,T]$ the map $y(t) \in C^{2, \mu}(\overline{\Omega}, \RR^d)$ satisfies the equations
    \begin{equation}
        \begin{cases}\label{equation for total deformation}
        -\Div(\mathcal{A}_{G(t)}(\nabla y(t))=0_d & \textup{on } \Omega
        \\ y(t) = f & \textup{on } \Gamma_D
        \\ \mathcal{A}_{G(t)}(\nabla y(t))\n = g & \textup{on } \Gamma_N,
        \end{cases}
    \end{equation}
    if for all $t \in [0,T]$ the map $N(t) \in C^{2, \mu}(\overline{\Omega}, \RR^d)$ satisfies the equations
    \begin{equation}\label{equation nutrient equation}
        \begin{cases}
            -\sum_{i,j=1}^d \partial_i(D(G(t),\nabla y(t))_{ij}\partial_j N(t))+\beta(G(t), \nabla y(t)) N(t)=0 & \textup{on } \Omega
            \\ N(t)=f_n & \textup{on } \Gamma_D^n
            \\ D(G(t),\nabla y(t))\nabla N(t)\cdot \n=g_n & \textup{on } \Gamma_N^n
        \end{cases}
    \end{equation}
    and if $G\in C^1([0,T];C^{1,\mu}(\overline{\Omega},\RR^{d \times d}))$ satisfies the ordinary differential equation 
    \begin{equation}\label{equation ODE}
        \begin{cases}
        \frac{d}{dt} G(t) = \mathcal{G}(G(t), \nabla y(t), N(t), \cdot), & t \in [0,T],
        \\ G(0) = G_0
        \end{cases}
    \end{equation}
    on the space $C^{1,\mu}(\overline{\Omega},\RR^{d \times d})$.
\end{definition}

\begin{remark}
    To emphasize that the growth tensor $G$ is the basic variable governing the evolution of the system, we will write, for $t \in [0,T]$, and a solution $(G,y,N)$ in the sense of Definition \ref{definition definition of solution} $y(G(t))=y_{G(t)}:=y(t)$ and similarly $N(G(t))=N_{G(t)}:=N(t)$.
\end{remark}

\begin{theorem}\label{theorem main theorem}
    There exist $R_G,R_y,R_f,R_g>0$ with the following property: For all $G_0 \in B_{R_G/2}(\mathbf{1}_d) \\ \subset C^{1,\mu}(\overline{\Omega},\RR^{d \times d})$, $f \in B_{R_f}(\id)\subset C^{2,\mu}(\Gamma_D,\RR^d)$ and $g \in B_{R_g}(0_d)\subset C^{1,\mu}(\Gamma_N,\RR^d)$, there exists a time horizon $T>0$ such that there exist a triple $(G,y,N)$ which is a local solution to the morphoelastic growth problem on the time horizon $[0,T]$ in the sense of Definition \ref{definition definition of solution}. Moreover, for all $t \in [0,T]$, $G(t) \in B_{R_G}(\mathbf{1}_d)$ and $y(t) \in B_{R_y}(\id)$.
\end{theorem}

\begin{remark}
    Time-dependent loads and time-dependent traction forces can be treated by standard methods as long as they are sufficiently small. A time-dependent displacement on the Dirichlet part of the boundary can be considered as well as long as the displacement on the boundary is close to the identity. The same holds for the nutrient equation, in which we can treat time-dependent Dirichlet and Neumann boundary conditions and a time-dependent source term as well.
\end{remark}

\begin{example}\label{example (realistic choice)}
    The following example is covered by our theory. We choose $f=\id$ on $\Gamma_D$ and $g=0$ on $\Gamma_N$. For $F \in \GLp{d}$, we choose
    \begin{equation*}
        W(\cdot,F)=\textup{dist}(F,\SO{d})^2 + \frac{1}{\det(F)^2} +\det(F)^2 - 2.
    \end{equation*}
    Let $\gamma \in C^2(\overline{\Omega})$, $\eta \in C^3(\RR_{\geq 0})$ and $\mu \in C^3(\RR^{d \times d})$. By 
    \begin{equation*}
        P(Y,G):=\det(G)D_pW(\cdot,Y G^{-1})G^{-T},
    \end{equation*}
    we denote the first Piola-Kirchhoff stress tensor associated to the growth tensor $G$ and a deformation gradient $Y=\nabla y(G)$. Moreover, we choose $G_0 = \id$ and, for $(G,Y,N) \in \GLp{d} \times \GLp{d} \times \RR_{\geq 0}$,
    \begin{equation*}
        \mathcal{G}(G,Y,N,x)=\gamma(x)\eta(N)\mu(P(Y,G))G.
    \end{equation*}
    Note, that our choice of $W$ is smooth on a sufficiently small ball around the identity and therefore satisfies Assumption \textbf{(G1)}. Assumptions $\textbf{(G2)}$ and \textbf{(G3)} are satisfied for obvious reasons. It remains to state an example for the absorption rate and the diffusion rate. By $D_0 \in C^{2,\mu}(\overline{\Omega},\RR^{d \times d}_{\sym})$ we denote the diffusion rate and by $\beta_0 \in C^{0,\mu}(\overline{\Omega})$ we denote the absorption rate in the unstressed reference configuration. For $G,Y \in \GLp{d}$, cf. Remark \ref{remark pointwise definition coefficients nutrient equation} for the pointwise definition of the coefficients, a possible choice, which is compliant with the principle of frame indifference, cf. Remark \ref{remark frame indifference 2}, is given by
    \begin{equation*}
        D(G,Y)=\frac{\det(G)}{\det(Y)}D_0 \quad\textup{and} \quad\beta(G,Y)=\frac{\det(Y)}{\det(G)}\beta_0.
    \end{equation*}
    This example accounts for the modeling assumption that elastic compression decreases the diffusion rate whereas it increases the absorption rate. Moreover, the choices for $W$, $\mathcal{G}$, $D$ and $\beta$ in this example are compliant with the principle of frame-indifference, cf. Remarks \ref{remark frame indifference} and \ref{remark frame indifference 2}.
\end{example}

\section{Auxiliary Results}

We will apply Picard-Lindel\"of's theorem on the Banach space $C^{1,\mu}(\overline{\Omega},\RR^{d \times d})$ to prove existence of a unique trajectory in a sufficiently small neighbourhood of the map congruent to the constant unit matrix for the growth tensor determining the growth dynamics of our model.

\begin{theorem}[Theorem 2.13;\cite{Schechter.2005}]\label{Theorem Picard Lindelof Banach space}
    Let $(X,\norm{\cdot})$ be a Banach space, $x_0\in X$, $R_0,T_0>0$, $t_0 \in \RR$, and let
    \begin{equation*}
        B_0=\{x \in X |\, \norm{x-x_0}\leq R_0\}
    \end{equation*}
    and
    \begin{equation*}
        I_0=\{t \in \RR |\, |t-t_0|\leq T_0\}.
    \end{equation*}
    Assume that $g: I_0 \times B_0 \to X$ is continuous and that there exist constants $K_0,M_0>0$ such that for all $x,y \in B_0$ and for all $t \in I_0$
    \begin{equation*}
        \norm{g(t,x)-g(t,y)}\leq K_0\norm{x-y}
    \end{equation*}
    and
    \begin{equation*}
        \norm{g(t,x)}\leq M_0.
    \end{equation*}
    Let $T_1$ be such that
    \begin{equation*}
        T_1\leq \min(T_0,\frac{R_0}{M_0}),\quad K_0T_1<1.
    \end{equation*}
    Then, there exists a unique solution $x(t) \in B_0$ of the ordinary differential equation
    \begin{equation*}
        \begin{cases}
            \frac{d}{dt}x(t)=g(t,x(t)) \quad \textup{for all } t \textup{ with } |t-t_0|\leq T_1
            \\ x(t_0)=x_0.
        \end{cases}
    \end{equation*}
\end{theorem}

The following theorem is a direct consequence from the convergence of the Neumann series for operators on Banach spaces with norm smaller than one.

\begin{proposition}\label{Proposition Neumann series argument for uniform boundedness of inverse}
    Let $X,Y$ be Banach spaces and let $T:X\to Y$ be a linear and invertible operator. If for $S:X \to Y$ and $\lambda \in (0,1)$
    \begin{equation*}
        \norm{S-T}\leq \lambda \frac{1}{\norm{T}},
    \end{equation*}
    the operator $S$ is invertible as well and $\norm{S^{-1}}\leq \frac{1}{1-\lambda}\norm{T^{-1}}$. 
\end{proposition}

The following propsition shows that composition is well-behaved in H\"older spaces.

\begin{proposition}\label{proposition composition in Holder spaces}
    Let $m \in \NN_0$, $n,N \in \NN$, $\mu \in (0,1]$, $U \subset \RR^n$ open and let $F \in C^{m+1}(\overline{\Omega} \times U, \RR^N)$. Then, for all $R>0$ and $K \subset U$ compact, there exists a constant $C(R,K)>0$ such that for all $u \in C^{m,\mu}(\overline{\Omega},\RR^n)$ with $\norm{u}_{C^{m,\mu}(\overline{\Omega},\RR^n)} \leq R$ and for which, for all $x \in \overline{\Omega}$, $u(x) \in K$, there holds
        \begin{equation*}
            \norm{F(\cdot, u(\cdot))}_{C^{m,\mu}(\overline{\Omega},\RR^N)} \leq C(R,K).
        \end{equation*}
        If, moreover, $D_2F \in C^{m+1}(\overline{\Omega} \times U, \RR^{n \times N})$, then for every $R>0$ and $K \subset U$ compact, there exists $C(R,K)>0$ such that for all $u, v \in C^{m,\mu}(\overline{\Omega},\RR^n)$ with $\norm{u}_{C^{m,\mu}(\overline{\Omega},\RR^n)},\norm{v}_{C^{m,\mu}(\overline{\Omega},\RR^n)} \leq R$ and with, for all $x \in \overline{\Omega}$ and $t \in [0,1]$, $u(x) + t(v(x) - u(x)) \in K$, there holds
        \begin{equation*}
            \norm{F(\cdot,u(\cdot))-F(\cdot,v(\cdot))}_{C^{m,\mu}(\overline{\Omega},\RR^N)} \leq C(R,K)\norm{u-v}_{C^{m,\mu}(\overline{\Omega},\RR^n)}.
        \end{equation*}
    \end{proposition}

    \begin{proof}
        The proof of the first assertion can be done in a straightforward manner. The second assertion is proved as follows. By the mean-value theorem
        \begin{align*}
            & \norm{F(\cdot,u(\cdot))-F(\cdot,v(\cdot))}_{C^{m,\mu}(\overline{\Omega},\RR^N)} 
            \\ &=  \norm{\int_0^1 D_2 F(\cdot,u(\cdot) + t(v(\cdot)-u(\cdot)))[u(\cdot)-v(\cdot)] dt}_{C^{m,\mu}(\overline{\Omega},\RR^N)}
            \\&\leq  C\int_0^1 \norm{D_2 F(\cdot,u(\cdot) + t(v(\cdot)-u(\cdot)))}_{C^{m,\mu}(\overline{\Omega},\RR^{n \times N})}dt\norm{u-v}_{C^{m,\mu}(\overline{\Omega},\RR^n)}
        \end{align*}
        for some constant $C>0$ dependent only on $m$. As $D_2F \in C^{m+1}(\overline{\Omega} \times U,\RR^{n \times N})$, the first statement yields the second inequality.
    \end{proof}

We review a generalized Korn inequality which applies to the case that the symmetric part of the gradient is pertubed by a mapping that is sufficiently regular and close to the unit matrix. The following theorem is due to Pompe and can be found in \cite{Pompe.2003}. It is based on earlier work by Neff, who proved it in \cite{Neff.2002} in case the pertubation is in $C^2$ and $d=3$, neither of which is necessarily the case here. As we exploit uniformity of the constant in the pertubation, we cite a version of the statement from \cite{Mielke.Roubicek.2016}, also cf. Theorem 3.5 in \cite{Neff.2005} for a slightly less general variant. 

\begin{proposition}[Corollary 3.4; \cite{Mielke.Roubicek.2016}]\label{proposition Korn Pompe Neff}
    For $\lambda \in (0,1)$ and $K>1$ we define the set
    \begin{equation*}
        F_{K}:=\{F \in C^{0,\lambda}(\overline{\Omega},\RR^{d \times d}):\, \norm{F}_{C^{0,\lambda}(\overline{\Omega},\RR^{d \times d})}\leq K,\, \min_{x \in \overline{\Omega}} \det (F(x)) \geq \frac{1}{K}\}.
    \end{equation*}
    Then there exists a constant $C(K)>0$ such that for all $F \in F_K$ and for all $v \in H^1(\Omega,\RR^d )$ with $v \vert_{\Gamma_D}=0$
    \begin{equation*}
        \int_{\Omega} |F^T\nabla v +(\nabla v)^T F|^2dx \geq C(K)\norm{v}_{H^1(\Omega,\RR^d)}^2.
    \end{equation*}
\end{proposition}

Moreover, we review some results on the existence, uniqueness, and regularity of solutions of linear elliptic systems of partial differential equations.

\begin{definition}\label{definition Legendre-Hadamard ellipticity}
    Let for $\alpha, \beta, i,j \in \{1, \dots , d\}$ the coefficients $A_{ij}^{\alpha\beta} \in L^{\infty}(\Omega)$. We say that the operator $A \in L^{\infty}(\Omega,\RR^{d^4})$, defined component-wise via $A^{\alpha  \beta}_{ij}$, is uniformly elliptic in the Cauchy-Hadamard sense if there exists a constant $\nu>0$ independent of $x \in \Omega$ such that for all $p,q \in \RR^d$ and for almost all $x \in \Omega$
    \begin{equation*}
        A_{ij}^{\alpha\beta}(x)q^i q^j p_{\alpha} p_{\beta} \geq \nu |p|^2|q|^2.
    \end{equation*}
\end{definition}

\begin{proposition}\label{Proposition Coercivity implies ellipticity}
    Let, for $\alpha, \beta, i,j \in \{1, \dots ,d\}$, $A_{ij}^{\alpha\beta} \in L^{\infty}(\Omega)$ and assume that for some constant $\nu>0$ for all $v \in H_0^1(\Omega,\RR^d)$ the coercivity condition
    \begin{equation}\label{coercivity condition elliptic PDE}
        \int_{\Omega} A_{ij}^{\alpha\beta}\partial_{\beta} v^i \partial_{\alpha} v^j dx \geq \nu \int_{\Omega} |\nabla v|^2 dx
    \end{equation}
    is satisfied. Then, for almost every $x \in \Omega$ the Legendre-Hadamard condition is satisfied, i.e., for all $p,q \in \RR^d$ and for almost all $x \in \Omega$
    \begin{equation*}
        A_{ij}^{\alpha\beta}(x)q^i q^j p_{\alpha} p_{\beta} \geq \nu |p|^2|q|^2.
    \end{equation*}
\end{proposition}

\begin{proof}
    Let $p,q \in \RR^d$, let $\phi \in C^{\infty}_0(\overline{\Omega})$ and $\lambda>0$. The claim follows by taking $u \in H_0^1(\Omega,\RR^d)$ defined via $v(x):=q \phi(x)\sin(\lambda p \cdot x)$ as a test function in \eqref{coercivity condition elliptic PDE} and letting $\lambda$ converge to $\infty$.
\end{proof}

The following theorem is a direct consequence of Lax-Milgram's theorem and the regularity theory developed in \cite{Agmon.Douglis.Nirenberg.1959, Agmon.Douglis.Nirenberg.1964, Douglis.Nirenberg.1955}, cf. in particular Theorem 9.3 in \cite{Agmon.Douglis.Nirenberg.1964}.

\begin{theorem}\label{Theorem existence, regularity and uniqueness linear elliptic systems}
    Let $A=(A_{ij}^{\alpha\beta})_{ij\alpha\beta}\in C^{1,\mu}(\overline{\Omega}, \RR^{d^4})$ be an operator for which the coercivity condition \eqref{coercivity condition elliptic PDE} in Proposition \ref{Proposition Coercivity implies ellipticity} is satisfied and $f \in C^{0,\mu}(\overline{\Omega},\RR^d)$.
    Then, there exists a unique solution $v\in C^{2,\mu}(\overline{\Omega},\RR^d)$ to the boundary value problem
    \begin{equation*}
        \begin{cases}
            -\partial_{\beta}(A_{ij}^{\alpha\beta}\partial_\alpha v^j)=f_i &\textup{ on } \Omega
            \\ v=0_d &\textup{ on } \Gamma_D
            \\ A_{ij}^{\alpha\beta}\partial_{\alpha} v^j \n_{\beta}=0 &\textup{ on } \Gamma_N
        \end{cases}
    \end{equation*}
    for $i\in \{1,\dots, d\}$ and there exists a constant $C>0$ depending only on $\Omega, \Gamma_D, \Gamma_N, \mu, d$ and $A$ such that
    \begin{equation*}
        \norm{v}_{C^{2,\mu}(\overline{\Omega},\RR^d)}\leq C\norm{f}_{C^{0,\mu}(\overline{\Omega},\RR^d)}.
    \end{equation*}
\end{theorem}

\section{Existence and Lipschitz Dependence of Solutions to the quasi-static Momentum Balance Equation}

In this section, we prove existence of stationary solutions via contraction mapping arguments. We start with some notation and general remarks.

\begin{remark}\label{remark shifted operator}
  Let $f \in C^{2,\mu}(\Gamma_D,\RR^d)$ and $g \in C^{1,\mu}(\Gamma_N,\RR^d)$. Moreover, let, for $T:C^{2,\mu}(\overline{\Omega},\RR^d) \to C^{2,\mu}(\Gamma_D,\RR^d)$ denoting the trace operator, $E: C^{2,\mu}(\Gamma_D,\RR^d) \to C^{2,\mu}(\overline{\Omega},\RR^d)$ be a continuous right-inverse and $C_T:= \norm{E}_{\op}$. We define $\tilde{f}:=\id + E(f-\id)$. Then, 
  \begin{equation*}
      \norm{\tilde{f}-\id}_{C^{2,\mu}(\overline{\Omega},\RR^d)} \leq C_T\norm{f-\id}_{C^{2,\mu}(\Gamma_D,\RR^d).}
  \end{equation*}
  Therefore, if $f$ is close to the identity, $\tilde{f}$ is close to the identity as well. In particular, if \\ $\norm{f-\id}_{C^{2,\mu}(\Gamma_D,\RR^d)} \leq \frac{1}{(d+1)C_T}$, there exists a constant $c>0$ such that $\det(\nabla \tilde{f})>c$ on $\overline{\Omega}$. Let $R_G^1 \in (0,R)$ as in \textbf{(W1)} and $G \in B_{R_G^1}(\mathbf{1}_d)$. Moreover, let $R_{\nabla y}^2 \in (0,R)$ and let $R_f^1>0$ be sufficiently small. We define the operator $\mathcal{A}_G^0:B_{R_{\nabla y}^2}(\mathbf{0}_d)\subset C^{1,\mu}(\overline{\Omega},\RR^{d \times d}) \to C^{1,\mu}(\overline{\Omega},\RR^{d \times d}) \times C^{1,\mu}(\Gamma_N,\RR^d)$ via
    \begin{equation}\label{equation shifted operator}
        \mathcal{A}_G^0(U)= \begin{bmatrix}
            \mathcal{A}_G^{0,1}(U) \\ \mathcal{A}_G^{0,2}(U)
        \end{bmatrix}:=\begin{bmatrix}
            \det(G)D_pW(\cdot, (U+\nabla \tilde f)G^{-1})G^{-T}
            \\\det(G)D_pW(\cdot, (U+\nabla \tilde f)G^{-1})G^{-T}\n-g
        \end{bmatrix}.
    \end{equation}
    Moreover, we define $\Div_1:C^{1,\mu}(\overline{\Omega},\RR^{d \times d}) \times C^{1,\mu}(\Gamma_N,\RR^d) \to C^{0,\mu}(\overline{\Omega},\RR^d) \times C^{1,\mu}(\Gamma_N,\RR^d)$ via 
    \begin{equation*}
        \Div_1\biggl(\begin{bmatrix}
        \Phi \\ \varphi
    \end{bmatrix}\biggr):=\begin{bmatrix}
        \Div(\Phi)\\ \varphi
    \end{bmatrix}.
    \end{equation*}
    A map $y \in C^{2,\mu}(\overline{\Omega},\RR^d)$ with $y=f$ on $\Gamma_D$ and $\mathcal{A}_G(\nabla y)\n=g$ on $\Gamma_N$ satisfies $-\Div(\mathcal{A}_G(\nabla y))=0_d$, cf. \eqref{equation operator quasi-static momentum balance equation}, if and only if the map $u:=y-\tilde{f}$ satisfies the equation
    \begin{equation}\label{equation shifted operator quasi-static momentum balance equation}
        -\Div\mathcal{A}^{0,1}_G(\nabla u) =0_d
    \end{equation}
    together with the boundary conditions $u=0_d$ on $\Gamma_D$ and $\mathcal{A}^{0,2}_G(\nabla u)=0_d$ on $\Gamma_N$.
\end{remark}

\begin{lemma}\label{Lemma operator governing the elastic deformation is twice continuously Frechet differentiable}
    Let $R_G^1 \in (0,R)$ as in \textbf{(W1)} and $G \in B_{R_G^1}(\mathbf{1}_d)$. Moreover, let $R_y^1 \in (0,R)$ be sufficiently small and let $R_f^1>0$ be as in Remark \ref{remark shifted operator}. For $f \in B_{R_f^1}(\id)$, the operator $\mathcal{A}^0_G$ as defined in \eqref{equation shifted operator} is twice continuously Fréchet-differentiable and its linearization around a point $\nabla u_0 \in C^{1,\mu}(\overline{\Omega}, \RR^{d \times d})$ with $u_0 \in B_{R_y^1}(\id) \subset C^{2,\mu}(\overline{\Omega},\RR^d)$, $D\mathcal{A}_G^0(\nabla u_0): C^{1,\mu}(\overline{\Omega}, \RR^{d \times d}) \to C^{1, \mu}(\overline{\Omega}, \RR^{d\times d}) \times C^{1,\mu}(\Gamma_N,\RR^d)$, is given pointwise by 
    \begin{equation*}
        \Phi \mapsto D\mathcal{A}^0_G(\nabla u_0)[\Phi]=\begin{bmatrix}
            \det(G) D_p^2W(\cdot, \nabla (u_0+ \tilde{f}) G^{-1})[\Phi G^{-1}]G^{-T}
            \\ \det(G) D_p^2W(\cdot, \nabla (u_0+ \tilde{f}) G^{-1})[\Phi G^{-1}]G^{-T}\n
        \end{bmatrix}.
    \end{equation*}
    Moreover, for $\Phi, \Psi \in C^{1,\mu}(\overline{\Omega}, \RR^{d \times d})$
    \begin{equation*}
        D^2\mathcal{A}^0_G(\nabla u_0)[\Phi, \Psi] = \begin{bmatrix}
            \det(G) D_p^3W(\cdot, \nabla (u_0+ \tilde{f}) G^{-1})[\Psi G^{-1}, \Phi G^{-1}]G^{-T}
            \\ \det(G) D_p^3W(\cdot, \nabla (u_0+ \tilde{f}) G^{-1})[\Psi G^{-1}, \Phi G^{-1}]G^{-T}\n
        \end{bmatrix}
        .
    \end{equation*}
\end{lemma}

\begin{proof}
    The proof is a direct consequence of Assumption \textbf{(W1)} and the definition of $\mathcal{A}_G^0$, cf. \eqref{equation shifted operator}.
\end{proof}

Moreover, for $R_G=R>0$ with $R$ as in \textbf{(W1)} and $G \in B_{R_G}(\mathbf{1})$, we define the linear operator $\mathcal{L}_G:C^{2,\mu}_{0,\Gamma_D}(\overline{\Omega},\RR^d) \to C^{0,\mu}(\overline{\Omega},\RR^d) \times C^{1,\mu}(\Gamma_N,\RR^d)$, which coincides with the linearization of $-\Div_1(\mathcal{A}_G^0(\nabla \cdot))$ in the point $0_d \in C^{2,\mu}_{0,\Gamma_D}(\overline{\Omega},\RR^d)$, via
\begin{equation}\label{equation definition linearized operator}
    \mathcal{L}_G[\varphi]=
    \begin{bmatrix}
        \mathcal{L}_G^1[\varphi] \\ \mathcal{L}_G^2[\varphi]
    \end{bmatrix}
    := \begin{bmatrix}
        -\Div(\det(G) D_p^2W(\cdot, \nabla (0_d+ \tilde{f}) G^{-1})[\nabla \varphi G^{-1}]G^{-T})
        \\ -\det(G) D_p^2W(\cdot, \nabla (0_d+ \tilde{f}) G^{-1})[\nabla \varphi G^{-1}]G^{-T}\n
    \end{bmatrix}.
\end{equation}

\begin{proposition}\label{proposition the linearized operator is a homeomorphism between Banach spaces}
    There exist $R_G^2,R_f^2 >0$ such that for all $G \in B_{R_G^2}(\mathbf{1}_d)$ and $f \in B_{R_f^2}(\id)$, the operator $\mathcal{L}_G:C_{0,\Gamma_D}^{2, \mu}(\overline{\Omega}, \RR^d) \to C^{0, \mu}(\overline{\Omega}, \RR^d) \times C^{1,\mu}(\Gamma_N,\RR^d)$ is a Banach space isomorphism. Moreover, the operator norm of its inverse, $\mathcal{L}_G^{-1}:C^{0,\mu}(\overline{\Omega},\RR^d) \times C^{1,\mu}(\Gamma_N,\RR^d) \to C^{2,\mu}_{0,\Gamma_D}(\overline{\Omega},\RR^d)$, is uniformly bounded in $G \in B_{R_G^2}(\mathbf{1}_d)$ and $f \in B_{R_f^2}(\id)$, i.e., there exist constants $C_{\mathcal{L}}>0$ and $C_{\mathcal{L}^{-1}}>0$ such that for all $G\in B_{R_G^2}(\mathbf{1}_d)$ and for all $f \in B_{R_f^2}(\id)$
    \begin{equation*}
        \norm{\mathcal{L}_G}_{\mathcal{L}(C^{2,\mu}_{0,\Gamma_D}(\overline{\Omega},\RR^d),C^{0,\mu}(\overline{\Omega},\RR^d) \times C^{1,\mu}(\Gamma_N,\RR^d))}\leq C_{\mathcal{L}}
    \end{equation*}
    and
    \begin{equation*}
        \norm{\mathcal{L}_G^{-1}}_{\mathcal{L}(C^{0,\mu}(\overline{\Omega},\RR^d) \times C^{1,\mu}(\Gamma_N,\RR^d),C^{2,\mu}_{0,\Gamma_D}(\overline{\Omega},\RR^d))}\leq C_{\mathcal{L}^{-1}}.
    \end{equation*}
\end{proposition}

\begin{proof}
    Let $G \in B_{R_G^2}(\mathbf{1}_d)$ with $R_G^2$ sufficiently small. Then, taking into account assumption \textbf{(W1)},
    \begin{align*}
        & \norm{\mathcal{L}_G}_{\mathcal{L}(C^{2,\mu}_{0,\Gamma_D}(\overline{\Omega},\RR^d),C^{0,\mu}(\overline{\Omega},\RR^d) \times C^{1,\mu}(\Gamma_D,\RR^d))}
        \\ &=  \sup_{\varphi \in C^{2,\mu}_{0,\Gamma_D}(\overline{\Omega},\RR^d); \norm{\varphi}_{C^{2,\mu}_{0,\Gamma_D}(\overline{\Omega},\RR^d)}=1} (\norm{-\Div(\det(G) D_p^2W(\cdot, \nabla (0_d+ \tilde{f}) G^{-1})[\nabla \varphi G^{-1}]G^{-T})}_{C^{0,\mu}(\overline{\Omega},\RR^d)}
        \\ &\quad + \norm{-\det(G) D_p^2W(\cdot, \nabla (0_d+ \tilde{f}) G^{-1})[\nabla \varphi G^{-1}]G^{-T})\n}_{C^{1,\mu}(\Gamma_N,\RR^d)})
        \\ &\leq  C\sup_{\varphi \in C^{2,\mu}_{0,\Gamma_D}(\overline{\Omega},\RR^d); \norm{\varphi}_{C^{2,\mu}_{0,\Gamma_D}(\overline{\Omega},\RR^d)}=1} \norm{\det(G) D_p^2W(\cdot, \nabla (0_d+ \tilde{f}) G^{-1})[\nabla \varphi G^{-1}]G^{-T}}_{C^{1,\mu}(\overline{\Omega},\RR^{d \times d})}
        \\ &\leq  C_{\mathcal{L}}.
    \end{align*}
    We proceed by showing that the operator $\mathcal{L}_G$ is elliptic in the sense of Definition \ref{definition Legendre-Hadamard ellipticity}. For $\varphi \in C_0^{2, \mu}(\overline{\Omega}, \RR^d)$, the coefficients $A_{ij}^{\alpha\beta}:\Omega\to\RR$ of the equation $\mathcal{L}_G^1[\varphi]=0$
    are given by
    \begin{equation*}
        A_{ij}^{\alpha\beta}=\sum_{l,m=1}^d \det(G)D_p^2W(\cdot,\nabla (0_d+ \tilde{f})G^{-1})_{ij}^{lm}(G^{-1})^{\beta}_m (G^{-1})^{\alpha}_l.
    \end{equation*} 
    Solving $\mathcal{L}_G^1[\varphi]=0$ is equivalent to solving
    \begin{equation*}
        -\biggl(\sum_{\alpha,\beta,j=1}^d \partial_{\alpha} (A_{ij}^{\alpha\beta}\partial_\beta \varphi^j)\biggr)_{i=1,\dots, d}=0_d.
    \end{equation*}
    According to Proposition~\ref{Proposition Coercivity implies ellipticity}, it suffices to show that there exists $\nu>0$, such that for all $\psi\in H_0^1(\Omega,\RR^d)$
    \begin{equation*}
        \int_{\Omega} \sum_{\alpha,\beta,i,j=1}^n A_{ij}^{\alpha\beta}\partial_\beta \psi^j \partial_\alpha \psi^i dx \geq \nu\int_{\Omega} |\nabla \psi|^2dx.
    \end{equation*}
    According to Assumption \textbf{(W1)} and Taylor's theorem, cf. Theorem 4.A in \cite{Zeidler.1986}, we can write for $A \in B_{R}(\mathbf{1}_d)$ and $B \in \GLp{d}$ and some constant $C_1>0$
    \begin{align*}
        & D_p^2W(\cdot,A)[B,B]=  D_p^2W(\cdot,\mathbf{1}_d)[B,B] + D_p^3W(\cdot,\mathbf{1}_d)[A-\mathbf{1}_d,B,B]+O(|A-\mathbf{1}_d|^2|B|^2)
        \\ & \geq  D_p^2W(\cdot,\mathbf{1}_d)[B,B] - C_1 |A-\mathbf{1}_d||B|^2+O(|A-\mathbf{1}_d|^2|B|^2).
    \end{align*}
    Moreover, if $B \in B_{R}(\mathbf{0}_d)$, we obtain by Taylor's theorem and \textbf{(W2)}, that for $\varepsilon \in (0,1)$
    \begin{align*}
        & W(\cdot,\mathbf{1}_d+\varepsilon B) =  W(\cdot,\mathbf{1}_d) + \varepsilon D_pW(\cdot,\mathbf{1}_d)[B] +\frac{\varepsilon^2}{2} D_p^2W(\cdot,\mathbf{1}_d)[B,B] + O(\varepsilon^3|B|^3)
        \\&  =  \frac{\varepsilon^2}{2} D_p^2W(\cdot,\mathbf{1}_d)[B,B] + O(\varepsilon^3|B|^3).
    \end{align*}
    Due to Assumption \textbf{(W2)}
    \begin{equation*}
        c_W\textup{dist}(\mathbf{1}_d+\varepsilon B,\SO{d})^2 \leq \frac{\varepsilon^2}{2} D_p^2W(\cdot,\mathbf{1}_d)[B,B] + O(\varepsilon^3|B|^3).
    \end{equation*}
    We compute the polar decomposition of $\mathbf{1}_d+\varepsilon B$ and obtain by linearization around the identity, cf. (3.20) in \cite{Friesecke.James.Mueller.2002}, that
    \begin{equation*}
        \textup{\dist}(\mathbf{1}_d+\varepsilon B,\SO{d})= \frac{\varepsilon}{2} |B+B^T| + O(\varepsilon^2|B|^2).
    \end{equation*}
    Dividing by $\varepsilon$ and letting $\varepsilon$ converge to zero, we obtain that
    \begin{equation*}
        |B+B^T|^2 \leq \frac{2}{c_W}D_p^2W(\cdot,\mathbf{1}_d)[B,B].
    \end{equation*}
    Let $\psi \in C_c^{\infty}(\Omega,\RR^d)$ such that $\norm{\nabla \psi G^{-1}}_{L^{\infty}(\Omega,\RR^{d \times d})} < R$. By Proposition \ref{proposition Korn Pompe Neff}, for a constant $C>0$ depending solely on $\Omega$, $c_W$ and the Korn constant $C(K)>0$, we can estimate
    \begin{align*}
        & \int_{\Omega}\sum_{\alpha,\beta,i,j=1}^n A_{ij}^{\alpha\beta}\partial_\beta \psi^j \partial_\alpha \psi^i dx
        \\ &= \int_{\Omega} \det(G)D_p^2W(\cdot,\nabla(\mathbf{0}_d+ \tilde{f})G^{-1})[\nabla \psi G^{-1},\nabla \psi G^{-1}]dx
        \\ &\geq  \int_{\Omega} \det(G)D_p^2W(\cdot,\mathbf{1}_d)[\nabla \psi G^{-1},\nabla \psi G^{-1}]-C_1\det(G) |\nabla(\mathbf{0}_d+ \tilde{f})G^{-1}-\mathbf{1}_d||\nabla \psi G^{-1}|^2 
        \\ &\quad + O(|\nabla(\mathbf{0}_d+ \tilde{f})G^{-1}-\mathbf{1}_d|^2|\nabla \psi G^{-1}|^2) dx
        \\ &\geq  \int_\Omega \frac{c_W}{2}\det(G)|(\nabla \psi G^{-1})^T+\nabla \psi G^{-1}|^2-C_1 \det(G) |\nabla \tilde{f}G^{-1}-\mathbf{1}_d||\nabla \psi G^{-1}|^2  
        \\ &\quad + O(|\nabla(\mathbf{0}_d+ \tilde{f})G^{-1}-\mathbf{1}_d|^2|\nabla \psi G^{-1}|^2) dx
        \\ &\geq  C \int_{\Omega} |\nabla \psi|^2 -C_1 |\nabla\tilde{f}G^{-1}-\mathbf{1}_d||\nabla \psi G^{-1}|^2  + O(|\nabla(\mathbf{0}_d+ \tilde{f})G^{-1}-\mathbf{1}_d|^2|\nabla \psi G^{-1}|^2) dx
        \\ &\geq  C\int_{\Omega} |\nabla \psi|^2dx.
    \end{align*}
    In the last step, we took into account that, for sufficiently small $R_G^2, R_f^2$, $G \in B_{R_G^2}(\mathbf{1}_d)$ and $f \in B_{R_G^2}(\mathbf{1}_d)$, $\tilde{f}$ as well as $G^{-1}$ are close to the identity by construction of $\tilde{f}$, cf. Remark \ref{remark shifted operator}. 
    By a scaling argument and approximation by compactly supported and smooth functions, we obtain that there exists a constant $\nu>0$ such that for all $\psi \in H_0^1(\Omega,\RR^d)$
    \begin{equation*}
        \int_{\Omega} \sum_{\alpha,\beta,i,j=1}^d A_{ij}^{\alpha\beta}\partial_\beta \psi^j \partial_\alpha \psi^i dx\geq \nu\int_{\Omega} |\nabla \psi|^2dx. 
    \end{equation*}
    Therefore, according to Theorem \ref{Theorem existence, regularity and uniqueness linear elliptic systems}, $\mathcal{L}_G:C^{2,\mu}_{0,\Gamma_D}(\overline{\Omega},\RR^{d \times d})\to C^{0,\mu}(\overline{\Omega},\RR^{d \times d})\times C^{1,\mu}(\Gamma_N,\RR^d)$ is bijective. The statement follows by Proposition \ref{Proposition Neumann series argument for uniform boundedness of inverse} which guarantees that, for sufficiently small $R_G^2>0$, the operator norm of the inverse is bounded uniformly in $G \in B_{R_G^2}(\mathbf{1}_d)$ and $f \in B_{R^2_f}(\id)$.
\end{proof}

Now, we want to prove existence of stationary solutions with the contraction principle. We first reformulate the equation \eqref{equation shifted operator quasi-static momentum balance equation} as a fixed-point problem. As the existence of $\mathcal{L}_G^{-1}$, for $\mathcal{L}_G$ as in \eqref{equation definition linearized operator}, is guaranteed according to Proposition \ref{proposition the linearized operator is a homeomorphism between Banach spaces}, rearranging yields that for $u \in B_{R_y^2}(0_d) \subset C^{2,\mu}_{0,\Gamma_D}(\overline{\Omega},\RR^d)$
\begin{align}\nonumber
    -\Div_1(\mathcal{A}^0_G(\nabla u)) = \begin{bmatrix}
        0_d \\ 0_d
        \end{bmatrix} &\iff \mathcal{L}_G[u]-\Div_1(\mathcal{A}^0_G(\nabla u)) = \mathcal{L}_G[u] \\\label{equation equation equivalent to fixed point problem}
  &\iff  \mathcal{L}_G^{-1}[\mathcal{L}_G[u]-\Div_1(\mathcal{A}^0_G(\nabla u))] = u.
\end{align}

We start by showing that the operator $\mathcal{L}_G^{-1}[\mathcal{L}_G(\cdot)-\Div(\mathcal{A}_G^0)(\nabla \cdot)]$ restricted to a sufficiently small closed ball around zero is a contraction.

\begin{lemma}\label{lemma contraction}
    For $R_G^3, R_f^3>0$ sufficiently small, there exists a constant $R_y^3>0$ such that for all $G \in B_{R_G^3}(\mathbf{1}_d)$ and for all $f \in B_{R_f^3}(\id)$ the operator $\mathcal{L}_G^{-1}[\mathcal{L}_G-\Div(\mathcal{A}^0_G(\nabla \cdot))]: B_{R_y^3}(0_d) \to C^{2, \mu}(\overline{\Omega}, \RR^d)$ is a contraction, i.e., $\mathcal{L}_G^{-1}[\mathcal{L}_G-\mathcal{A}^0_G(\nabla \cdot)]$ is Lipschitz continuous with Lipschitz constant strictly smaller than $1$. Moreover, the choice of $R_y^3$ is uniform in $R_G^3$ and $R_f^3$ for $R_G^3$ and $R_f^3$ sufficiently small, that is, for all $\tilde{R}_G^3 \in (0,R_G^3)$, $\tilde{R}_f^3 \in (0,R_f^3)$, $\tilde{R}_y^3 \in (0,R_y^3)$, for all $G \in B_{\tilde{R}_G^3}(\mathbf{1}_d)$ and for all $f \in B_{\tilde{R}_f^3}(\id)$, $\mathcal{L}_G^{-1}[\mathcal{L}_G-\mathcal{A}^0_G(\nabla \cdot)]$ is a contraction on $B_{\tilde{R}_y^3}(0_d)$.
\end{lemma}

\begin{proof}
    Let $R_G^3, R_f^3, R_y^3>0$ be sufficiently small, $f \in B_{R_f^3}(\id)$, $G \in B_{R_G^3}(\mathbf{1}_d)$ and $u_1, u_2 \in B_{R_y^3}(0_d)$. In particular, for all $u \in B_{R_y^2}(0_d) \subset C^{2,\mu}_{0,\Gamma_D}(\overline{\Omega},\RR^{d \times d})$
    \begin{equation*}
        \mathcal{L}_G[u] = - \Div_1(D\mathcal{A}_G^0(\nabla u_0)[\nabla u]).
    \end{equation*}
    According to Proposition \ref{proposition the linearized operator is a homeomorphism between Banach spaces}, if $R_G^3$ and $R_f^3$ are sufficiently small,
    \begin{equation*}
        \norm{\mathcal{L}_G^{-1}}_{\mathcal{L}(C^{0,\mu}(\overline{\Omega},\RR^d) \times C^{1,\mu}(\Gamma_N,\RR^d),C^{2,\mu}_{0,\Gamma_D}(\overline{\Omega},\RR^d))}\leq C_{\mathcal{L}^{-1}}
    \end{equation*}
    uniformly in $G \in B_{R_G^3}(\mathbf{1}_d)$ and $f \in B_{R_f^3}(\id)$. By computing the Taylor expansion of $\mathcal{A}_G^0$ around $\nabla u_0= \mathbf{0}_d$, that is, $u_0=0$, we can write for $\mathcal{R}_G:C^{1,\mu}(\overline{\Omega},\RR^{d \times d}) \to C^{1,\mu}(\overline{\Omega},\RR^{d \times d}) \times C^{1,\mu}(\Gamma_D,\RR^d)$ 
    \begin{equation*}
        \mathcal{A}_G^0(\nabla u)=\mathcal{A}_G^0(\nabla u_0) + D\mathcal{A}_G^0(\nabla u_0)[\nabla u - \nabla u_0] + \mathcal{R}_G(\nabla u)[ \nabla u - \nabla u_0,\nabla u - \nabla u_0]
    \end{equation*}
    with 
    \begin{equation*}
        \mathcal{R}_G(\nabla u)[\nabla u - \nabla u_0,\nabla u - \nabla u_0] = \begin{bmatrix} \mathcal{R}_G^1(\nabla u)[\nabla u - \nabla u_0,\nabla u - \nabla u_0]\\ \mathcal{R}_G^2(\nabla u)[\nabla u - \nabla u_0,\nabla u - \nabla u_0]\end{bmatrix}
    \end{equation*} 
    and
    \begin{equation}\label{equation relation between R1 and R2}
        \mathcal{R}_G^2(\nabla u)[\nabla u - \nabla u_0,\nabla u - \nabla u_0]=\mathcal{R}_G^1(\nabla u)[\nabla u - \nabla u_0,\nabla u - \nabla u_0]\n.
    \end{equation}
    Here, we omitted $\nabla u_0$ as an argument in the remainder for notational simplicity. Therefore, we can estimate by \eqref{equation relation between R1 and R2} and Proposition \ref{proposition composition in Holder spaces}, due to which $\mathcal{R}_G$ is bounded on $B_{R_y^3}(0_d)$ uniformly in $G \in B_{R_G^3}(\mathbf{1}_d)$ and $f \in B_{R_f^3}(\id)$,
    \begin{align*}
        & \norm{\mathcal{L}_G^{-1}[\mathcal{L}_G(u_1)-\Div_1(\mathcal{A}^0_G(\nabla u_1))]-\mathcal{L}_G^{-1}[\mathcal{L}_G(u_2)-\Div_1(\mathcal{A}^0_G(\nabla u_2))]}_{C^{2, \mu}(\overline{\Omega}, \RR^d)}
        \\ &\leq  C_{\mathcal{L}^{-1}} \norm{\mathcal{L}_G(u_1)-\Div_1(\mathcal{A}^0_G(\nabla u_1))-(\mathcal{L}_G(u_2)-\Div_1(\mathcal{A}^0_G(\nabla u_2)))}_{C^{0, \mu}(\overline{\Omega}, \RR^d) \times C^{1,\mu}(\Gamma_d,\RR^d)}
        \\ &=  C_{\mathcal{L}^{-1}} \lVert\Div_1(\mathcal{R}_G(\nabla u_2)[\nabla (u_2 - u_0),\nabla (u_2 -  u_0)]
         \\ & \quad - \mathcal{R}_G(\nabla u_1)[\nabla( u_1 - u_0),\nabla( u_1 -  u_0)])\Vert_{C^{0, \mu}(\overline{\Omega}, \RR^d) \times C^{1,\mu}(\Gamma_d,\RR^d)}
        \\ &\leq   C_{\mathcal{L}^{-1}} \lVert\mathcal{R}_G(\nabla u_2)[\nabla (u_2 - u_0),\nabla (u_2 - u_0)]
       \\ & \quad - \mathcal{R}_G(\nabla u_1)[\nabla (u_1 - u_0),\nabla (u_1 -  u_0)]\rVert_{C^{1, \mu}(\overline{\Omega}, \RR^{d \times d}) \times C^{1,\mu}(\Gamma_d,\RR^d)}
        \\ &\leq   C \norm{\mathcal{R}_G^1(\nabla u_2)[\nabla (u_2 -u_0),\nabla( u_2 - u_0)]- \mathcal{R}_G^1(\nabla u_1)[\nabla (u_1 -  u_0),\nabla (u_1 - u_0)]}_{C^{1, \mu}(\overline{\Omega}, \RR^{d \times d})}
        \\ &\leq   C \norm{(\mathcal{R}_G^1(\nabla u_2)-\mathcal{R}_G^1(\nabla u_1))[\nabla (u_2 - u_0),\nabla u_2 - \nabla u_0]}_{C^{1, \mu}(\overline{\Omega}, \RR^{d \times d})} 
        \\ &\quad + \norm{\mathcal{R}_G^1(\nabla u_1)[\nabla (u_2 - 2 u_0 +  u_1),\nabla( u_2 -  u_1)]}_{C^{1, \mu}(\overline{\Omega}, \RR^{d \times d})}
        \\ &\leq   C \norm{\mathcal{R}_G^1(\nabla u_2)-\mathcal{R}_G^1(\nabla u_1)}_{C^{1, \mu}(\overline{\Omega}, \RR^{d^6})} \norm{\nabla (u_2 -  u_0)}_{C^{1, \mu}(\overline{\Omega}, \RR^{d \times d})}^2
        \\ &\quad + \norm{\mathcal{R}_G^1(\nabla u_1)}_{C^{1, \mu}(\overline{\Omega}, \RR^{d^6})}\norm{\nabla (u_2 - 2 u_0 +  u_1)}_{C^{1, \mu}(\overline{\Omega}, \RR^{d \times d})}\norm{\nabla (u_2 -  u_1)}_{C^{1, \mu}(\overline{\Omega}, \RR^{d \times d})}
        \\ &\leq  C ((R_y^3)^2 + R_y^3)\norm{\nabla (u_2 - u_1)}_{C^{1, \mu}(\overline{\Omega}, \RR^{d \times d})}
    \end{align*}
    which becomes small for $R_y^3>0$ sufficiently small. The second assertion follows directly from the proof.
\end{proof}

In the next step we prove that under appropriate small strain assumptions, the operator $\mathcal{L}_G^{-1}[\mathcal{L}_G-\Div_1(\mathcal{A}^0_G(\nabla \cdot))]$ is a self-mapping.

\begin{lemma}\label{lemma self mapping}
    Let $R_G^3, R_f^3$ and $R_y^3$ be as in Lemma \ref{lemma contraction}. Then, there exist constants $R_G^4 \in (0,R_G^3)$, $R_f^4 \in (0,R_f^3)$, $R_y^4 \in (0,R_y^3)$ and $R_g^4>0$ such that for all $G\in B_{R_G^4}(\mathbf{1}_d)$, $f \in B_{R_f^4}(\id)$ and for all $g \in B_{R_g^4}(0_d)$ 
    \begin{equation}\label{equation self mapping}
        \mathcal{L}_G^{-1}[\mathcal{L}_G-\Div_1(\mathcal{A}^0_G(\nabla \cdot))](\overline{B}_{R_y^4}(0_d)) \subseteq \overline{B}_{R_y^4}(0_d).
    \end{equation}
    Moreover, $R_y^4>0$ and $R_f^4>0$ can be chosen arbitrarily small if the other radii are sufficiently small.
\end{lemma}

\begin{proof}
    Let $g \in C^{1,\mu}(\Gamma_N,\RR^d)$. According to Proposition \ref{proposition the linearized operator is a homeomorphism between Banach spaces} and the definition of $R_G^3$ and $R_f^3$, the operator $\mathcal{L}_G^{-1}$ exists for all $G \in B_{R_G^3}(\mathbf{1}_d)$ and $f \in B_{R_f^3}(\id)$ and its norm is bounded by $C_{\mathcal{L}^{-1}}$ uniformly in $G \in B_{R_G^3}(\mathbf{1}_d)$ and $f \in B_{R_f^3}(\id)$. Furthermore, $\mathcal{A}^0_G$ is twice continuously Fréchet-differentiable according to Lemma \ref{Lemma operator governing the elastic deformation is twice continuously Frechet differentiable}. Therefore, we can compute the Taylor expansion of $\mathcal{A}^0_G$ around $\nabla u_0:= \mathbf{0}_d$ that is, for $u_0=0_d$. We can estimate, for $u \in B_{R_y^3}(0_d)$,     
    \begin{align*}
        &\norm{\mathcal{L}_{G}^{-1}[\mathcal{L}_{G}(u)-\Div_1(\mathcal{A}^0_G(\nabla u))]-\mathbf{0}_d}_{C^{2,\mu}(\overline{\Omega}, \RR^d)}
        \\ &\leq   C_{\mathcal{L}^{-1}}\norm{\mathcal{L}_{G}(u)-\Div_1(\mathcal{A}^0_G(\nabla u))-\mathcal{L}_G(u_0)}_{C^{0, \mu}(\overline{\Omega},\RR^d) \times C^{1,\mu}(\Gamma_N,\RR^d)}
        \\ &\leq  C(\norm{\Div_1(\mathcal{A}^0_G(\nabla u_0))}_{C^{0,\mu}(\overline{\Omega},\RR^{d \times d}) \times C^{1,\mu}(\Gamma_N,\RR^d)}+\norm{g}_{C^{1,\mu}(\Gamma_N,\RR^d)}
        \\ &\quad+\norm{\Div_1(\mathcal{R}_G(\nabla u)[\nabla (u- u_0),\nabla( u-u_0)])}_{C^{0,\mu}(\overline{\Omega},\RR^d) \times C^{1,\mu}(\Gamma_N,\RR^d)})
        \\ &\leq  C(\norm{\mathcal{A}^0_G(\nabla u_0)}_{C^{1,\mu}(\overline{\Omega},\RR^{d \times d})}+\norm{g}_{C^{1,\mu}(\Gamma_N,\RR^d)}+\norm{\mathcal{R}^2_G(\nabla u)[\nabla (u- u_0),\nabla (u-u_0)]}_{C^{1,\mu}(\overline{\Omega},\RR^{d \times d})})
        \\ &\leq  C(\norm{\mathcal{A}^0_G(\nabla u_0)}_{C^{1,\mu}(\overline{\Omega},\RR^{d \times d})}+\norm{g}_{C^{1,\mu}(\Gamma_N,\RR^d)}+\norm{\mathcal{R}^2_G(\nabla u)}_{C^{0,\mu}(\overline{\Omega},\RR^{d^6})}\norm{\nabla (u-  u_0)}_{C^{1,\mu}(\overline{\Omega},\RR^{d \times d})}^2)
        \\ &\leq  C(\norm{\mathcal{A}^0_G(\nabla u_0)}_{C^{1,\mu}(\overline{\Omega},\RR^{d \times d})}+\norm{g}_{C^{1,\mu}(\Gamma_N,\RR^d)}+\norm{\mathcal{R}^2_G(\nabla u)}_{C^{0,\mu}(\overline{\Omega},\RR^{d^6})}\norm{u-  u_0}_{C^{2,\mu}(\overline{\Omega},\RR^d)}^2).
    \end{align*}
    We first choose $R_y^4 \in (0, R_y^3)$ sufficiently small such that for all $u \in \overline{B}_{R_y^4}(0_d)$, for all $G \in B_{R_G^3}(\mathbf{1}_d)$ and for all $f \in B_{R_f^3}(\id)$ 
    \begin{equation*}
        C\norm{\mathcal{R}^1_G(\nabla u)}_{C^{0,\mu}(\overline{\Omega},\RR^{d \times d})}\norm{u-  u_0}_{C^{2,\mu}(\overline{\Omega},\RR^d)}^2 \leq \frac{R_y^4}{2}.
    \end{equation*}
    We conclude by choosing $R_G^4 \in (0,R_G^3)$, $R_f^4 \in (0, R_f^3)$ and $R_g^4>0$ sufficiently small such that for all $G \in B_{R_G^4}(\mathbf{1}_d)$, $f \in B_{R_f^4}(\id)$ and for all $g \in B_{R_g^4}(0_d)$
    \begin{equation*}
        C(\norm{\mathcal{A}^0_G(u_0)}_{C^{1,\mu}(\overline{\Omega},\RR^{d \times d})}+\norm{g}_{C^{1,\mu}(\Gamma_N,\RR^d)}) \leq \frac{R_y^4}{2}.
    \end{equation*}
    The fact that $R_y^4>0$ and $R_f^4>0$ can be chosen arbitrarily small follows directly from the proof.
\end{proof}

\begin{corollary}\label{corollary existence and uniqueness of stationary solutions}
    For $R_G^4,R_f^4,R_g^4>0$ as in Lemma \ref{lemma self mapping}, there exists $R_y^5>0$ such that for all $f \in B_{R_f^4}(\id)$, $g \in B_{R_g^4}(0_d)$ and for all $G \in B_{R_G^4}(\mathbf{1}_d)$ there exists a unique solution $y(G) \in B_{R_y^5}(\id)$ of the equation
    \begin{equation}\label{equation momentum balance on Banach space}
        \begin{cases}
            -\Div(\mathcal{A}_G(\nabla y(G)))=0_d & \textup{on } \Omega
            \\ y(G)=f & \textup{on } \Gamma_D
            \\ \mathcal{A}_G(\nabla y(G))\n = g & \textup{on } \Gamma_N.
        \end{cases}
    \end{equation}
    Moreover, $R_y^5$ can be chosen arbitrarily small, if we restrict ourselves to smaller choices of $R_y^4$ and $R_f^4$ in Lemma \ref{lemma self mapping}. 
\end{corollary}

\begin{proof}
    According to Lemma \ref{lemma contraction}, for $R_G^4,R_f^4,R_g^4>0$ and $R_y^4>0$ as in Lemma \ref{lemma self mapping} and for all $G \in B_{R_G^4}(\mathbf{1}_d)$, $f \in B_{R_G^4}(\id)$ and for all $g \in B_{R_g^4}(0_d)$, the operator $\mathcal{L}_G^{-1}[\mathcal{L}_G-\Div(\mathcal{A}^0_G(\nabla \cdot))]: \overline{B}_{R_y^4}(0_d) \to C^{2, \mu}(\overline{\Omega}, \RR^d)$ is a contraction and, according to Lemma \ref{lemma self mapping}, a self-mapping. By Banach's fixed-point theorem, there exists a unique fixed-point $u \in B_{R_y^4}(0_d)$, that is
    \begin{equation*}
      u=  \mathcal{L}_G^{-1}[\mathcal{L}_G[u]-\Div(\mathcal{A}^0_G(\nabla u))].
    \end{equation*}
    By Remark \ref{remark shifted operator}, $y(G):= u+ \tilde{f}$ satisfies \eqref{equation momentum balance on Banach space} and, by construction of $\tilde{f}$, cf. Remark \ref{remark shifted operator}, for $R_y^5:=R_y^4 + C_TR_f^4$, $y(G) \in B_{R_y^5}(\id)$. According to the construction of $R_y^5$ and Lemma \ref{lemma self mapping}, $R_y^5$ can be chosen arbitrarily small. 
\end{proof}

\begin{proposition}\label{proposition deformation depends Lipschitz continuously on growth tensor}
    For $R_G^4, R_y^5, R_f^4, R_g^4>0$ as in Corollary \ref{corollary existence and uniqueness of stationary solutions}, cf. also Lemma \ref{lemma self mapping}, $f \in B_{R_f^4}(\id)$, $g \in B_{R_g^4}(0_d)$ and $G \in B_{R_G^4}(\mathbf{1}_d)$, we define $y(G) \in B_{R_y^5}(0_d)$ as the unique solution of \eqref{equation momentum balance on Banach space}, which exists according to Corollary \ref{corollary existence and uniqueness of stationary solutions}. The map $\mathfrak{y}: B_{R_G^4}(\mathbf{1}_d) \to B_{R_y^5}(\id)$ defined via $G \mapsto y(G)$ is of the class $C^1$. Moreover, there exist constants $L_{\mathfrak{y}}$ and $M_{\mathfrak{y}}$ such that, for some $R_G^7 \in (0,R_G^4)$, which can by chosen arbitrarily small,
    \begin{equation*}
        \sup_{G\in B_{R_G^7}(\mathbf{1}_d)}\norm{\frac{\partial}{\partial G}\mathfrak{y}(G)}_{\mathcal{L}(C^{1,\mu}(\overline{\Omega},\RR^{d \times d}),C^{2,\mu}_{0,\Gamma_D}(\overline{\Omega},\RR^{d \times d}))}\leq L_{\mathfrak{y}}
    \end{equation*}
    and
    \begin{equation}\label{equation deformation gradient bounded in G}
        \sup_{G \in B_{R_G^7}(\mathbf{1}_d)}\norm{\mathfrak{y}(G)}_{C^{2,\mu}(\overline{\Omega},\RR^d)}\leq M_{\mathfrak{y}}.
    \end{equation}
    In particular, the map $\mathfrak{y}:B_{R_G^7}(\mathbf{1}_d) \to B_{R_y^5}(\id)$ is Lipschitz continuous with Lipschitz constant $L_{\mathfrak{y}}$, i.e., for all $G_1,G_2 \in B_{R_G^7}(\mathbf{1}_d)$
    \begin{equation}\label{equation deformation gradient Lipschitz in time}
        \norm{\mathfrak{y}(G_1)-\mathfrak{y}(G_2)}_{C^{2,\mu}(\overline{\Omega},\RR^d)} \leq L_{\mathfrak{y}} \norm{G_1-G_2}_{C^{0,\mu}(\overline{\Omega},\RR^{d \times d})}.
    \end{equation}
\end{proposition}

\begin{proof}
    We apply the implicit function theorem on Banach spaces to show the statement for the map $G \mapsto y(G)-\tilde{f}$. We define the map $\mathcal{A}^0: C^{1,\mu}(\overline{\Omega},\RR^{d \times d}) \times C^{2,\mu}_{0,\Gamma_D}(\overline{\Omega},\RR^{d}) \to C^{1,\mu}(\overline{\Omega},\RR^{d \times d}) \times C^{1,\mu}(\Gamma_N,\RR^d)$ via  
    \begin{equation*}
        \mathcal{A}^0(G,u):=\begin{bmatrix}
            \det(G)D_pW(\cdot, (\nabla u+\nabla \tilde f)G^{-1})G^{-T}
            \\\det(G)D_pW(\cdot, (\nabla u+\nabla \tilde f)G^{-1})G^{-T}\n-g
        \end{bmatrix}=\begin{bmatrix}
            \mathcal{A}^{0,1}(G,u)
            \\\mathcal{A}^{0,2}(G,u)
        \end{bmatrix}.
    \end{equation*}
    As discussed in Remark \ref{remark shifted operator} and in Corollary \ref{corollary existence and uniqueness of stationary solutions}, for $G \in B_{R_G^4}(\mathbf{1}_d)$, there exists a unique solution $u \in C^{2,\mu}_{0,\Gamma_D}(\overline{\Omega},\RR^d)$ with $u \in B_{R_y^4}(0_d)$ of the equation
    \begin{equation*}
        -\Div_1(\mathcal{A}^0(G, u)) = \begin{bmatrix}
            0_d \\ 0_d
        \end{bmatrix}.
    \end{equation*}
    According to Proposition \ref{proposition the linearized operator is a homeomorphism between Banach spaces}, for all $G \in B_{R_G^4}(\mathbf{1}_d)$
    \begin{equation*}
        \mathcal{L}_G[\cdot]= -\Div_1(D_{u}\mathcal{A}^0(G,0_d)[\cdot])
    \end{equation*}
    is a linear homeomorphism between the space $C^{2,\mu}_{0,\Gamma_D}(\overline{\Omega},\RR^d)$ and $C^{0,\mu}(\overline{\Omega},\RR^d) \times C^{1,\mu}(\Gamma_N,\RR^d)$. Therefore, according to the implicit function theorem on Banach spaces, Theorem 4.B in \cite{Zeidler.1986}, there exists an open ball $B_{R_G^6}(\mathbf{1}_d)$ of radius $R_G^6 \in (0,R_G^4)$ and a function $\mathfrak{y} \in C^1(B_{R_G^6}(\mathbf{1}_d),B_{R_y^5}(\id))$ defined via $\mathfrak{y}:G \mapsto y(G)$. Moreover, by continuity, there exists $R_G^7 \in (0,R_G^6)$ such that 
    \begin{equation*}
        \sup_{G\in B_{R_G^7}(\mathbf{1}_d)}\norm{\frac{\partial}{\partial G}\mathfrak{y}(G)}_{\mathcal{L}(C^{1,\mu}(\overline{\Omega},\RR^{d \times d}),C^{1,\mu}(\overline{\Omega},\RR^{d \times d}))}\leq L_{\mathfrak{y}}.
    \end{equation*}
    and thus \eqref{equation deformation gradient bounded in G}. By the mean value theorem, \eqref{equation deformation gradient Lipschitz in time} holds. 
\end{proof}

\begin{remark}
    One can also show the weaker property that the deformation depends Lipschitz-continuously on the growth tensor directly without having to rely on the implicit function theorem, which would suffice for solving the ordinary differential equation \eqref{equation ODE}.
\end{remark}

\section{The Nutrient Equation}

We proceed by showing existence, regularity, Lipschitz regularity in time and non-negativity of solutions to the nutrient equation.

\begin{proposition}\label{proposition existence of solution to nutrient equation}
    For $R_G^4, R_y^5, R_f^4, R_g^4>0$ as in Corollary \ref{corollary existence and uniqueness of stationary solutions}, cf. also Lemma \ref{lemma self mapping}, $G \in B_{R_G^4}(\mathbf{1}_d)$, $f \in B_{R_f^4}(\id)$, $g \in B_{R_g^4}(0_d)$ and $y(G) \in B_{R_y^5}(\id)$ denoting the unique solution of \eqref{equation momentum balance on Banach space}, which exists according to Proposition \ref{corollary existence and uniqueness of stationary solutions}, there exists a unique solution $N(G) \in C^{2,\mu}(\overline{\Omega})$ of the nutrient equation
    \begin{equation}\label{equation nutrient equation on a Banach space}
        \begin{cases}
            -\sum_{i,j=1}^d \partial_i(D(G,\nabla y(G))_{ij}\partial_j N(G))+\beta(G,y(G))N(G)=0 & \textup{on } \Omega
            \\ N(G)=f_n & \textup{on } \Gamma_D^n
            \\ D(G,\nabla y(G)) \nabla N(G) \cdot \n=g_n & \textup{on } \Gamma_N^n.
        \end{cases}
    \end{equation}
    Moreover, the nutrient concentration is non-negative, that is, $N(G)\geq 0$ on $\overline{\Omega}$.
\end{proposition}

\begin{proof}
    Existence and uniqueness of a weak solution $N(G) \in H^1(\Omega)$ follows by Lax-Milgram's Theorem according to Assumptions \textbf{(N1)} and \textbf{(N2)}. Non-negativity of the solution follows by the weak maximum principle and Hopf's Theorem. H\"older regularity follows by classical Schauder theory.
\end{proof}

\begin{proposition}\label{proposition Lipschitz in time dependence of nutrients}
    For $R_G^4, R_y^5, R_f^4, R_g^4>0$ as in Corollary \ref{corollary existence and uniqueness of stationary solutions}, cf. also Lemma \ref{lemma self mapping}, $R_G^7>0$ as in Proposition~\ref{proposition deformation depends Lipschitz continuously on growth tensor}, $G \in B_{R_G^4}(\mathbf{1}_d)$, $f \in B_{R_f^4}(\id)$, $g \in B_{R_g^4}(0_d)$ and $y(G) \in B_{R_y^5}(\id)$ denoting the unique solution of \eqref{equation momentum balance on Banach space}, which exists according to Proposition \ref{corollary existence and uniqueness of stationary solutions}, we define $N(G)$ as the weak solution of the nutrient equation, cf.\ Proposition \ref{proposition existence of solution to nutrient equation}. There exist constants $M_N,L_N>0$ such that, for some $R_G^8 \in (0,R_G^7)$, the map $G \mapsto N(G)$ is continuously Fréchet differentiable as a mapping from $B_{R_G^8}(\mathbf{1}_d)$ to $C^{2,\mu}(\overline{\Omega})$ and 
    \begin{equation}\label{equation uniform boundedness nutrients}
        \sup_{G \in B_{R_G^8}(\mathbf{1}_d)} \norm{N(G)}_{C^{2,\mu}(\overline{\Omega})}\leq M_N
    \end{equation}
    and
    \begin{equation*}
        \sup_{G \in B_{R_G^8}(\mathbf{1}_d)}\norm{\frac{\partial}{\partial G}N(G)}_{\mathcal{L}(C^{1,\mu}(\overline{\Omega},\RR^{d \times d}),C^{2,\mu}(\overline{\Omega}))}\leq L_N.
    \end{equation*}
    In particular, the map $G \mapsto N(G)$ is Lipschitz continuous on $B_{R_G^8}(\mathbf{1}_d)$ with Lipschitz constant $L_N$, i.e., for all $G_1,G_2 \in B_{R_G^8}(\mathbf{1}_d)$
    \begin{equation}\label{equation Lipschitz continuity nutrients}
        \norm{N(G_1)-N(G_2)}_{C^{2,\mu}(\overline{\Omega})}\leq L_G\norm{G_1-G_2}_{C^{1,\mu}(\overline{\Omega},\RR^{d \times d})}.
    \end{equation}
\end{proposition}

\begin{proof}
    Analogously to the proof of Proposition \ref{proposition deformation depends Lipschitz continuously on growth tensor}, we apply the implicit function theorem. Let $\tilde{f}_n \in C^{2,\mu}(\overline{\Omega})$ with $\tilde{f}_n=f_n$ on $\Gamma_D^n$. Then, $N(G)$ is a solution of \eqref{equation nutrient equation on a Banach space} if and only if $N(G)^0:=N(G)-\tilde{f_n} \in C^{2,\mu}_{0,\Gamma_D^n}(\overline{\Omega})$ is a solution of
    \begin{equation*}
        \begin{cases}
            -\sum_{i,j=1}^d \partial_i(D(G,\nabla y(G))_{ij}\partial_j (N(G)^0 + \tilde{f_n})) + \beta(G,\nabla y)(N(G)^0 + \tilde{f_n})=0 & \textup{on } \Omega
            \\ N(G)^0=0 & \textup{on } \Gamma_D^n
            \\ D(G,\nabla y(G))\nabla (N(G)^0 + f_n)\cdot\n = g_n & \textup{on } \Gamma_N^n.
        \end{cases}
    \end{equation*}    
    We define $\mathcal{N}^0:B_{R_G}(\mathbf{1}_d) \times C^{2,\mu}(\overline{\Omega}) \to C^{0,\mu}(\overline{\Omega}) \times C^{1,\mu}(\Gamma_N^n)$ via
    \begin{equation*}
        (G,N^0) \mapsto \begin{bmatrix}
            -\sum_{i,j=1}^d \partial_i(D(G,\nabla y(G))_{ij}\partial_j (N^0 + \tilde{f_n})) + \beta(G,\nabla y)(N^0 + \tilde{f_n})
            \\  D(G,\nabla y(G))\nabla (N^0 + f_n)\cdot \n - g_n
        \end{bmatrix}
    \end{equation*}
    and apply the implicit function theorem, cf. Theorem 4.B in \cite{Zeidler.1986}. Clearly, for all $(G,N^0) \in B_{R_G^7}(\mathbf{1}_d) \times C^{2,\mu}_{0,\Gamma_D^n}(\overline{\Omega})$, $\frac{\partial}{\partial N} \mathcal{N}^0(G,N^0): C^{2,\mu}_{0,\Gamma_D^n}(\overline{\Omega}) \to C^{0,\mu}(\overline{\Omega}) \times C^{1,\mu}(\Gamma_N^n)$, given by
    \begin{equation*}
        \frac{\partial}{\partial N} \mathcal{N}^0(G,N^0)[\varphi]=\begin{bmatrix}
            -\sum_{i,j=1}^n \partial_i(D(G,\nabla y(G))_{ij}\partial_j \varphi ) + \beta(G,\nabla y)\varphi 
            \\  D(G,\nabla y(G))\nabla \varphi \cdot \n
        \end{bmatrix}
    \end{equation*}
    is a Banach space isomorphism. Therefore, the mapping $N^0:B_{R_G^7}(\mathbf{1}_d) \to C^{2,\mu}(\overline{\Omega})$ defined via $G \mapsto N^0(G)$ is of the class $C^1$. Moreover, by continuity, there exists $R_G^8 \in (0,R_G^7)$ such that
    \begin{equation*}\label{equation nutrients are in C1 in time}
        \sup_{G \in B_{R_G^8}(\mathbf{1}_d)}\norm{\frac{\partial}{\partial G}N(G)}_{\mathcal{L}(C^{1,\mu}(\overline{\Omega},\RR^{d \times d}),C^{2,\mu}(\overline{\Omega}))}\leq L_N.
    \end{equation*}
    Thus \eqref{equation uniform boundedness nutrients} follows and \eqref{equation Lipschitz continuity nutrients} holds due to the mean value theorem.
\end{proof}

\section{Local Existence of Solutions}

We are now in a position to prove Theorem \ref{theorem main theorem} via Picard-Lindel\"of's theorem, cf.\ Theorem \ref{Theorem Picard Lindelof Banach space}. 
\begin{proof}[Proof of Theorem \ref{theorem main theorem}]
    Let $R_G^8>0$ be as in Proposition \ref{proposition Lipschitz in time dependence of nutrients} and $R>0$ as in \textbf{(W1)}, cf. also $\textbf{(G1)}$, define $R_G^9:=\min \{R_G^8,R\}$, let $R_f:=R_f^4, R_g:=R_g^4>0$ as in Corollary \ref{corollary existence and uniqueness of stationary solutions}, let $f \in B_{R_f^4}(\id)$, $g \in B_{R_g^4}(0_d)$ and let $\hat{\mathcal{G}}: B_{R_G^9}(\mathbf{1}_d) \to C^{1,\mu}(\overline{\Omega},\RR^{d \times d})$ be defined via
    \begin{equation*}
        \hat{\mathcal{G}}(G):=\mathcal{G}(G, \nabla y (G), N(G),\cdot)
    \end{equation*}
    where $y(G) \in C^{2,\mu}(\overline{\Omega},\RR^d)$ denotes the unique solution of the quasi-static equilibrium equation, \eqref{equation momentum balance on Banach space}, which exists according to Corollary \ref{corollary existence and uniqueness of stationary solutions} and let $N(G) \in C^{2,\mu}(\overline{\Omega})$ denote the unique solution of the nutrient equation, \eqref{equation nutrient equation on a Banach space}, which exists according to Proposition \ref{proposition existence of solution to nutrient equation}. We show that for all $G_0 \in B_{R_G^9}(\mathbf{1}_d)$ there exists a time horizon $T>0$, such that the ordinary differential equation
    \begin{equation*}\label{ODE on Banach space}
        \begin{cases}
            \frac{d}{dt} G(t)=\hat{\mathcal{G}}(G(t)) \quad \textup{on } [0,T]
            \\ G(0)=G_0
        \end{cases}
    \end{equation*}
    has a solution $G \in C^1 ([0,T];B_{R_G^9}(\mathbf{1}_d))$. In particular, this implies that $y(t):=y(G(t)) \in C^{2,\mu}(\overline{\Omega},\RR^d)$ solves \eqref{equation for total deformation} and $N(t):=N(G(t)) \in C^{2,\mu}(\overline{\Omega})$ solves \eqref{equation nutrient equation}. Additionally, $y \in C^1([0,T];C^{2,\mu}(\overline{\Omega},\RR^d))$ and $N \in C^1([0,T];C^{2,\mu}(\overline{\Omega}))$. To apply Theorem \ref{Theorem Picard Lindelof Banach space} it suffices to show that there exist constants $K_0,M_0>0$ such that for all $G,G_1,G_2 \in B_{R_G^9}(\mathbf{1}_d)$
    \begin{equation*}
        \norm{\hat{\mathcal{G}}(G_1)-\hat{\mathcal{G}}(G_2)}_{C^{1,\mu}(\overline{\Omega},\RR^{d \times d})} \leq K_0 \norm{G_1-G_2}_{C^{1,\mu}(\overline{\Omega},\RR^{d \times d})}
    \end{equation*}
    and
    \begin{equation*}
        \norm{\hat{\mathcal{G}}(G)}_{C^{1,\mu}(\overline{\Omega},\RR^{d \times d})}\leq M_0.
    \end{equation*}
    This is a direct consequence of Assumption \textbf{(G1)} and Propositions \ref{proposition composition in Holder spaces}, \ref{proposition deformation depends Lipschitz continuously on growth tensor} and \ref{proposition Lipschitz in time dependence of nutrients} as for all $G_1,G_2 \in B_{R_G^9}(\mathbf{1}_d)$
    \begin{align*}
        & \norm{\hat{\mathcal{G}}(G_1)-\hat{\mathcal{G}}(G_2)}_{C^{1,\mu}(\overline{\Omega},\RR^{d \times d})} 
        \\ & \leq  C (\norm{G_1-G_2}_{C^{1,\mu}(\overline{\Omega},\RR^{d \times d})} + \norm{\nabla y(G_1)-\nabla y(G_2)}_{C^{1,\mu}(\overline{\Omega},\RR^{d \times d})}  +\norm{N(G_1)-N(G_2)}_{C^{1,\mu}(\overline{\Omega})})
        \\ & \leq  K_0 \norm{G_1-G_2}_{C^{1,\mu}(\overline{\Omega},\RR^{d \times d})}
    \end{align*}
    and for all $G \in B_{R_G^9}(\mathbf{1}_d)$
    \begin{equation*}
        \norm{\hat{\mathcal{G}}(G)}_{C^{1,\mu}(\overline{\Omega},\RR^{d \times d})} \leq C(\norm{G}_{C^{1,\mu}(\overline{\Omega},\RR^{d \times d})} + \norm{y(G)}_{C^{1,\mu}(\overline{\Omega},\RR^{d \times d})} + \norm{N(G)}_{C^{1,\mu}(\overline{\Omega},\RR^{d \times d})}) \leq M_0.
    \end{equation*}
    Theorem \ref{theorem main theorem} follows with $R_G:=R_G^9$ and Picard-Lindel\"of's theorem, Theorem \ref{Theorem Picard Lindelof Banach space}, applied to $\mathcal{\hat{G}}\vert_{B_{R_{G}/2}(G_0)}$.
\end{proof}

\begin{remark}
    Note that the $L^{\infty}$-norm of the growth tensor can blow up on a finite time interval due to the general structure of the ordinary differential equation \eqref{equation ODE}. This difficulty could be circumvented by imposing a specific structure on the ordinary differential equation, for example an exponential growth law as in \cite{Davoli.Nik.Stefanelli.2023}. If we choose, for $\mathcal{H} \in L^{\infty}(\GLp{d} \times \GLp{d} \times \RR_{\geq 0} \times \overline{\Omega})$, a growth law of the form
    \begin{equation*}
        \frac{d}{dt} G= G\mathcal{H}(G,\nabla y, N,\cdot),
    \end{equation*}
    this dampens the effect of the elastic stress by imposing a hard constraint on how much the stress and the nutrient concentration can impact the $L^{\infty}$-norm of the growth tensor and prevents a blow-up of the growth tensor. Moreover, this also leads to a uniform lower bound of the determinant of the growth tensor by Jacobi's rule, that is,
    \begin{equation*}
        \frac{d}{dt} \det(G)= \det(G) \tr(\mathcal{H}(G,\nabla y, N,\cdot))
    \end{equation*}
    and hence
    \begin{align*}
        \det(G(t)) = &\det(G_0)\exp\biggl(\int_0^t\tr(\mathcal{H}(G,\nabla y, N,\cdot)) ds\biggr) 
        \\ \geq & \det(G_0)\exp(-td\norm{\mathcal{H}}_{L^{\infty}(\GLp{d} \times \GLp{d} \times \RR_{\geq 0} \times \overline{\Omega})})>0 \quad \textup{on } \Omega.
    \end{align*}
    However, one of the main challenges in the large strain regime lies in the fact that a solution of \eqref{equation for total deformation}, even on a much larger space, is unknown to exist if the growth tensor forces the strains to be large. Furthermore, the linearization around a solution, if it exists, is not necessarily a Banach space isomorphism, which is the main structural barrier for the analysis.
\end{remark}

\section{Example}

Our analysis is complemented by an example, which satisfies Assumptions \textbf{(W1)}-\textbf{(W3)}, \textbf{(G1)} and \textbf{(N1)}-\textbf{(N2)}, for suitable time-dependent Dirichlet boundary conditions. An analogous example can be constructed for pure Neumann boundary conditions. Such examples could provide a benchmark for numerical schemes. 
\begin{example}\label{example}
    We define, for $p \geq 1$ and $F \in \GLp{d}$,
    \begin{equation*}
        W(F) := \textup{dist}(F,\SO{d})^2 + (\det(F)^p + \frac{1}{\det(F)^p}-2) 
    \end{equation*}
    and, for $f \in H^{\frac{1}{2}}(\partial \Omega,\RR^d)$, 
    \begin{equation*}
        \mathcal{A}_{f}:=\{y \in H^1(\Omega,\RR^d):\, y\vert_{\partial \Omega}=f, \, \det(\nabla y) > 0 \textup{ a.e. on } \Omega\}.
    \end{equation*}
    We choose $f$ such that $\mathcal{A}_f$ is non-empty. For example, we can take $f \in C^{2,\mu}(\partial \Omega, \RR^d)$ close to the identity, cf. the discussion in Remark \ref{remark shifted operator} regarding the positivity of the determinant. By the polar composition, we immediately see that $W \in C^{\infty}(\GLp{d},\RR_{\geq 0})$ as, for $F \in \GLp{d}$ and $R(F):=F(F^TF)^{-\frac{1}{2}}$, $\textup{dist}(F,\SO{d})^2=\norm{F-R(F)}_{\textup{Frob}}^2$. Moreover, for a given growth tensor $G \in C^{1,\mu}(\overline{\Omega},\RR^{d \times d})$, we define the energy $E_G:\mathcal{A}_{f} \to [0,+\infty)$ via
    \begin{equation*}
        E_G(y):=\int_{\Omega} W(\nabla y G^{-1})\det(G)dx.
    \end{equation*}
    If $G$ is compatible, that is, $G=\nabla g$ for some $g:\Omega \to \RR^d$, and moreover $g=f$ on $\partial \Omega$, $y=g$ is a minimizer of $E_G$ on $\mathcal{A}_f$ and therefore also a solution to the associated Euler-Lagrange equations
    \begin{equation*}
        -\Div(\det(G)D_pW(\nabla y G^{-1})G^{-T})=0.
    \end{equation*}
    We choose $G_0 \in C^{1,\mu}(\overline{\Omega},\RR^{d \times d})$ such that there exists $g_0:\overline{\Omega} \to \RR^d$ with $\nabla g_0=G_0$ with $g_0=\id$ on $\partial\Omega$. Moreover, we choose $\mathcal{G}:\GLp{d} \times \GLp{d} \times \RR_{\geq 0} \times \overline{\Omega} \to \RR^{d \times d}$ defined via
    \begin{equation*}
        \mathcal{G}(G,Y,N,x):=GY
    \end{equation*}
    which clearly satisfies Assumption \textbf{(G1)}. The solution $G \in C^1([0,T];C^{1,\mu}(\overline{\Omega},\RR^{d \times d}))$ to the ordinary differential equation is then given by $G(t,\cdot)=(1-t)^{-1}G_0$. Clearly, $G$ is also compatible and therefore, the total deformation is always stress free and given by the deformation due to pure growth. The total deformation $y \in C^1([0,T];C^{2,\mu}(\overline{\Omega},\RR^d))$ is therefore, for all $t \in [0,T]$, given by $y(t,\cdot)=g(t,\cdot)=(1-t)^{-1}g_0$. We obtain an example if we prescribe the time-dependent Dirichlet boundary data $f(t,\cdot):=(1-t)^{-1}g_0$ on $\partial\Omega$. Moreover, for some choices of $D$ and $\beta$, such as a constant diffusion and absorption rate, the solution of the nutrient equation can be computed explicitly in spherical domains.
\end{example}
In our example, we circumvent two difficulties of our problem. The equation governing the growth process does not depend on the nutrient concentration. Moreover, we assume that the growth tensor is compatible at all times. In our example, the deformation is stress-free at all times.

\vspace{\baselineskip}
\noindent
\textbf{Acknowledgement:}
  The second author was supported by the Graduiertenkolleg 2339 IntComSin of the Deutsche Forschungsgemeinschaft (DFG, German Research Foundation) – Project-ID 321821685.
\newline
\medskip

\noindent
\textbf{Declaration of interest:} The authors do not have any conflicts of interest to declare.


\end{document}